\documentclass[12pt]{amsart}
\usepackage{hyperref}
\usepackage{amsmath, amsthm,amstext}
\usepackage{amssymb, array, amsfonts}
\usepackage[english]{babel}
\usepackage{url}
\usepackage{amsrefs}
\usepackage{hyperref}
\usepackage{float}
\usepackage{graphics}
\usepackage{graphicx}
   \usepackage{enumerate}
\numberwithin{equation}{section}

\usepackage[margin=1.1in]{geometry}

\def \ep{\varepsilon}

\renewcommand{\l}{\left}
\renewcommand{\r}{\right}

\def \Q{\mathbb{Q}}
\def \rL{\mathrm L}

\def \N{\mathbb{N}}
\def \M2{\mathrm{M}_2}
\def \R{\mathbb{R}}
\def \Z{\mathbb{Z}}

\def \sl2r{\mathrm{SL}(2,\R)}

\newcommand{\beq}{\begin{equation}}
\newcommand{\eeq}{\end{equation}}

\def\diag{\operatorname{diag}}

\def\ran{\operatorname{ran}}

\def \one{\mathbf{1}}

\newcommand{\eqdef}{\stackrel{\rm def}{=\kern-3.6pt=}}

\theoremstyle{plain}
\newtheorem{theorem}{\bf Theorem}[section]
\newtheorem{lemma}[theorem]{\bf Lemma}

\newtheorem{prop}[theorem]{\bf Proposition}
\newtheorem{cor}[theorem]{\bf Corollary}

\theoremstyle{definition}

\theoremstyle{remark}
\newtheorem{remark}[theorem]{\bf Remark}

\renewcommand{\le}{\leqslant}
\renewcommand{\ge}{\geqslant}

\newcommand{\dist}{\mathop{\mathrm{dist}}\nolimits}

\renewcommand{\qed}{\vrule height7pt width5pt depth0pt}
\title[Sharp arithmetic localization for
monotone potentials]{Sharp arithmetic localization for quasiperiodic operators with monotone potentials}
\author[S. Jitomirskaya]{Svetlana Jitomirskaya}
\address{Department of Mathematics, University of California, Berkeley, CA 94720, USA.}
\email{sjitomi@berkeley.edu}

\author[I. Kachkovskiy]{Ilya Kachkovskiy}
\address{Department of Mathematics,
	Michigan State University,
	Wells Hall, 619 Red Cedar Road,
	East Lansing, MI 48824,
	United States of America}
\email{ikachkov@msu.edu}

\date{}
\begin{document}

\begin{abstract}We prove the universality of sharp arithmetic
  localization for all 
  one-dimensional quasiperiodic Schr\"odinger operators with
  anti-Lipschitz monotone
  potentials. 

\end{abstract}
\maketitle
\section{Introduction and main results}
For one-dimensional ergodic Schr\"odinger operator families acting on $\ell^2(\Z)$:
\beq
\label{eq_h_def}
(H(x)\psi)(n)=\psi(n+1)+\psi(n-1)+f(T^nx)\psi(n),
\eeq
where $T:M\to M$ is an ergodic map, $x\in M$ and $f:M\to \mathbb R,$
positivity of the Lyapunov exponent (see \eqref{thouless}) is often taken in
physics literature as a signature of localization.
Yet, spectral localization  (pure point spectrum with exponentially decaying
eigenfunctions) does not in general follow from the positivity of the
Lyapunov exponents, as was first shown in \cite{Gor, as}. General powerful
spectral corollaries of positive Lyapunov exponents include absence of
absolutely continuous spectrum through the Pastur-Ishii theorem
(e.g. \cite{CFKS}) and a.e. zero-dimensionality of spectral measures
\cite{simonspectral}.  However, distinguishing between pure point and
singular continuous spectrum requires a delicate study of the
interplay between the Lyapunov growth and {\it resonances}: box
restrictions with very close eigenvalues.

Fot quasiperiodic operators, with $T$ being an ergodic shift $x\mapsto
x+\alpha$ on a
torus $\mathbb T=\mathbb R/\mathbb Z,$ one kind of such resonances comes purely from the
frequency: if $\mathrm{dist}(q\alpha,\Z)<e^{-cq}$ for infinitely many $q,$ this creates infinitely many resonances of
exponential strength and interferes with proofs of localization,
thus requiring Diophantine conditions. We will call those
``frequency resonances''. For sufficiently large $c$ this
leads to absence of eigenvalues \cite{as}.

A pathway to exploit positivity of the Lyapunov exponents to prove
localization was originally developed in \cite{Lana94} in the context of
almost Mathieu operators. It was significantly developed into a robust
method by Bourgain
and collaborators, see \cite{bbook}, but in a non-arithmetic way. As
for the arithmetic results, it was conjectured in \cite{jcongr} that,
for the almost Mathieu operators, for (arithmetically) a.e. phase,
frequency resonances are the only type that appears, and that there is
a {\it sharp arithmetic transition} between localization and singular continuity: the spectrum is pure point when  $
L(E)>\beta(\alpha)$ and singular continuous when $
0<L(E)<\beta(\alpha),$ where $L(E)$ is the Lyapunov exponent at the
energy $E$ and $\beta(\alpha)$, defined by \eqref{beta}, is the measure of irrationality of the frequency.

For almost Mathieu operators, $L(E)$ does not depend on $E$ for
energies $E$ in the spectrum, and is a function of the coupling
constant only. Remarkably, the same arithmetic transition was
partially proved
earlier, in \cite{simmar}, for the Maryland model, a quasiperiodic
operator with $v(x)=\lambda\tan \pi x$ (see \cite{JL} for
the full proof of a more precise statement), where $L(E)$ does depend
on $E$, this becoming a transition inside the spectrum.

A long standing problem is the {\it universality} of the sharp arithmetic transition for
quasiperiodic operators. Here we prove such universality for 
quasiperiodic Schr\"odinger operators with monotone $f$ satisfying an
anti-Lipschitz condition.

It will be more convenient for us to deal with periodic functions on
$\mathbb R$ rather than functions on $\mathbb T.$ Thus we will study
\beq
\label{eq_h_def}
(H(x)\psi)(n)=\psi(n+1)+\psi(n-1)+f(x+n\alpha)\psi(n),\quad x\in \mathbb R
\eeq
with $\gamma$-monotone (``anti-Lipschitz'') periodic
functions $f$. More precisely, let $f$ be periodic

\beq
\label{eq_f1}
f\colon \R\to [-\infty,+\infty),\quad f(x+1)=f(x).
\eeq
For $\gamma>0,$ we will say that $f$ is {\it $\gamma$-monotone on $[0,1)$} if
\beq
\label{eq_f2}
f(y)-f(x)\ge \gamma(y-x),\quad \text{for}\quad 0\le x<y< 1.
\eeq
In particular, $f$ must be finite on $(0,1)$.

We will also require the standard integrability condition
\beq
\label{eq_f3}
\int_0^1\log (1+|f(x)|)\,dx<+\infty,
\eeq
which is needed for the Lyapunov exponent to be finite, similar to \cite{Kach}.

If $f$ is unbounded, the original expression \eqref{eq_h_def} only defines the operator for $x\in \R\setminus(\Z+\alpha \Z)$. However, one can naturally extend it to all values of $x$ by considering the operator at {\it infinite coupling}: if $f(x+m\alpha)=-\infty$, consider the subspace of $\ell^2(\Z)$ defined by $\psi(m)=0$, and only require \eqref{eq_h_def} to hold for $n\neq m$. The new operator will become a direct sum of two half-line operators with Dirichlet boundary conditions, and our results will apply to this case.

The operator family $\eqref{eq_h_def}$, parametrized by $x$, is a family of (possibly unbounded) quasiperiodic operators, a particular case of an ergodic operator family. Let $L(E)$ be the Lyapunov exponent of this family at the energy $E$ (see the definition in the next section). The goal of the present paper is to establish new sharp connections between the values $L(E)$, the measure of irrationality of $\alpha$, and the spectral type of the operators $H(\cdot)$. Assume that $\alpha$ is irrational, and $\frac{p_k}{q_k}$ is the sequence of continued fraction approximants of $\alpha$. The measure of irrationality of $\alpha$ is defined by
\beq\label{beta}
\beta(\alpha):=\limsup\limits_{k\to +\infty} \frac{\log q_{k+1}}{q_k}.
\eeq
The following is the main result of the paper.
\begin{theorem}
\label{th_main} Suppose the operator family \eqref{eq_h_def} satisfies $\eqref{eq_f1}$ -- $\eqref{eq_f3}$. Then the set
$$
\{E\in \R\colon L(E)>\beta(\alpha)\}
$$
can only support purely point component of the spectral measure, and every generalized eigenfunction of $H(x)$ with $E$ in the above set decays exponentially.
\end{theorem}

In a separate note \cite{jkgordon} we establish that the set $$
\{E\in \R\colon 0<L(E)<\beta(\alpha)\}
$$
can only support singular continuous spectrum for a.e. $x\in\R$, for
a class containing all $v$ of locally bounded variation satisfying \eqref{eq_f3}. This is the
most general sharp Gordon-type argument. This has
been previously proved for Lipschitz $v$ in \cite{AYZ} (see also \cite{martini},
as well as for the Maryland model \cite{simmar,JL}, and, more
generally, meromorphic potentials \cite{JY}.

Together, this establishes the universality of the sharp arithmetic
transition for operators  \eqref{eq_h_def} satisfying $\eqref{eq_f1}$
-- $\eqref{eq_f3}$.

Also, as in \cite{JK,Kach}, localization turns out to be uniform in the same sense, see Corollary \ref{cor_uniform_localization}.

The above results are, of course, only meaningful in the regime of positive Lyapunov exponents. It turns out to be the case for a large class of the potentials under consideration. First, one has uniform positivity for non-perturbatively large $\gamma$, see \cite{Kach}:
\begin{prop}
Assuming $\eqref{eq_f1}$ -- $\eqref{eq_f3}$, the Lyapunov exponent of the family \eqref{eq_h_def} satisfies the lower bound
$$
L(E)\ge \max\{0,\log(\gamma/2e)\}.
$$
\end{prop}
One also has the following almost everywhere positivity result as a
corollary of \cite{Damanik_Killip}\footnote{proving a conjecture of
  \cite{mz}.} and, in the unbounded case, \cite{Simon_Spencer}, with the density of states argument in \cite{JK,Kach}, which is particularly helpful for the case $\beta(\alpha)=0$. See also Corollary \ref{cor_almost_every_x} for further discussion.
\begin{prop}
Suppose the operator family \eqref{eq_h_def} satisfies $\eqref{eq_f1}$ -- $\eqref{eq_f3}$, and either $f$ is unbounded or has finitely many discontinuities on each trajectory of the irrational shift by $\alpha$. Then the set
$$
\{E\in \R\colon L(E)=0\}
$$
has zero Lebesgue measure and, for almost every $x$, also zero
spectral measure with respect to $H(x)$.

As a consequence of Theorem $\ref{th_main}$ and \cite{jkgordon}, under
the above assumptions, the operators $\eqref{eq_h_def}$ have, for almost
every $x,$ a sharp arithmetic transition:
\begin{enumerate}\item  Anderson localization on
  $\{E\in \R\colon L(E)>\beta(\alpha)\}$
  \item pure singular continuous spectrum on   $\{E\in \R\colon
    L(E)<\beta(\alpha)\}.$
    \end{enumerate}

\end{prop}

The key to establishing the universality of  sharp transition in the class of
monotonic potentials, and, more generally, in proving localization, is in showing that frequency resonances are the
only ones that appear.  The framework to proving localization for operators with monotone potentials was
developed in \cite{JK} for the bounded case, where localization for
small couplings was a surprise.  It was extended in
\cite{Kach} for the unbounded case, in the regime of no exponential
resonances, thus the Diophantine condition on $\alpha :
\beta(\alpha)=0.$ The results of \cite{JK} were extended (also to the underlying circle
map dynamic) to localization on the set $L(E)>C \beta(\alpha)$ in
\cite{Zhu}.

Sharp arithmetic transitions were established originally for the
Maryland model and  the almost Mathieu operator in 
\cite{simmar,JL, AYZ,JL2,JHY}. In a very recent \cite{gj} the universality
of sharp
arithmetic transition was
established for all operators \eqref{eq_h_def} with {\it analytic} $f$
that are of type I (an open set of {\it analytic} potentials introduced in \cite{gjy}, see
\cite{wm} for history). Here we 
prove the universality for anti-Lipschitz monotone potentials with
essentially no other conditions.

Proofs of sharp transitions fall, so far, into two categories: the
ones based on some sort of duality \cite{simmar,JL, AYZ, gjy} and the ones
based on Green function estimates that trace back to
\cite{Lana94}. The sharp way to treat frequency resonances was
developed in \cite{JL2}. It has been extended to sharp results for the
other types of resonances for
the almost Mathieu in \cite{JL2, Lresonant}, the Maryland model
in \cite{JHY}, and adapted to even acceleration $1$ potentials, for which
most of the almost Mathieu considerations apply almost directly, in
\cite{lpos}. 

Here we make the sharp treatment of frequency resonances ideas of
\cite{JL} work with the methods of \cite{JK,Kach}. In the cases where we
rely on \cite{Damanik_Killip} for a.e. positivity of the Lyapunov exponent,
compared to
previous results that use the ideas of \cite{JL2}, we lose the arithmetic
description of a.e. phase, but the latter is impossible in the generality
we obtain our results. In cases with uniformly positive Lyapunov
exponent, we obtain localization for {\it all} phases.

In the process, we
also significantly streamline several parts of the method of
\cite{JL2}, developing several tools and approaches that we believe
will become useful in future proofs of localization where the analysis
of exponential resonances (of various kinds) is involved.


After the large deviation estimates are obtained, we utilize the key  ideas
of \cite{JL2} for the sharp resonance treatment. In particular,  concepts and structure of
resonant/non-resonant zones and regular points/good intervals in our proof is
parallel to \cite{JL2}, albeit with some modifications. We would
sometimes use similar notation. In the case of general monotone
potentials, however, one cannot take advantage of any (even
approximate) specific trigonometric
polynomial structure of the determinants of the finite volume
operators. An approach to producing good intervals for monotone
potentials was developed in \cite{JK,Kach} in the Diophantine case (for
intervals of size $q_k$). Using these intervals, one can push the
arguments towards some range of values of $\beta$ as in \cite{Zhu},
but it appears not sufficient for the sharp transitions. We outline
below the key novel technical aspects of the present paper, compared
to \cite{JK, Kach,JL2}
\begin{enumerate}
        \item Our proof of the  large deviation-type Theorem
          \ref{th_ldt} (LDT) is directly based on the Thouless
          formula, and does not require any dynamical considerations.
	\item We obtain LDT
          for the restrictions of the operator \eqref{eq_h_def} on
          intervals of {\it any size} $n=sq_k+r$. For comparison, the
          LDTs of \cite{JK, Kach} were for intervals of size $q_k$
          only. It turns out that one can obtain meaningful results by treating the operator as a perturbation of $sq_k$ copies of the operator on the interval of size $q_k$, of rank $O(s+|r|)$.
	\item The main application of the large deviation theorem is an ability to produce a shift of a given interval so that the value of the determinant of the finite volume operator is not too small. For an interval of size $n$, this ultimately corresponds to considering the values of the determinant at $n+1$ points, obtained from the irrational rotation. A general Lagrange interpolation argument for the complete determinant in the same generality as \cite{JL2,JHY} does not appear to be available. However, by changing the size of the intervals in the non-resonant cases, one can create gaps in the set of sampling points, and ultimately apply Lagrange interpolation to a polynomial obtained from a cluster of sampling points (Lemma \ref{lemma_sampling_points}).
	\item We develop a streamlined version of obtaining
          localization as a consequence of existence of good intervals
          (this part is very general), based on the observation that the estimates of maximal values of the generalized eigenfunctions in resonant zones are somewhat similar to the block resolvent expansion. The inequalities between those maximal values can be interpreted as regularity of the generalized eigenfunction at resonant points (see \eqref{eq_regularity_psi} for the definition). The exponent in the definition of regularity depends on whether the point is resonant or not and, if it is resonant, on how far from the origin the resonance is located. Then we use the general fact (Lemmas \ref{lemma_dominating}, \ref{lemma_terminating}) that, ultimately, the decay of the solution is determined by the worst possible rate one can obtain through Poisson iterations, with the error from combinatorial factors and faster decaying terms negligible on the logarithmic scale.
\end{enumerate}

\section{Some preliminaries: spectral theory, ergodic operators, and irrational rotations}
\subsection{Preliminaries from the spectral theory of ergodic
  operators} In this subsection, we formulate some basic results from
the spectral theory of ergodic operators and establish immediate consequences for the model \eqref{eq_h_def}. Let 
$$
H_n(x):=\l.\one_{[0;n-1]}H(x)\r|_{\ran \one_{[0;n-1]}}
$$
be the $(n\times n)$-block of $H(x)$, or equivalently the Dirichlet restriction of the operator $H(x)$ into $[0;n-1]$.
Denote by $N_n(x,E)$ the counting function of the eigenvalues of $H_n(x)$: 

$$
N_n(x,E):=\#\l(\sigma(H_n(x))\cap [-\infty,E)\r).
$$
Since we are also considering the case of infinite coupling, let us assume that the corresponding infinite eigenvalue is always equal to $-\infty$ for the purposes of the counting function. The {\it integrated density of states} of the operator family $H(x)$ is defined as
\beq
\label{ids_def}
N(E)=\lim\limits_{n\to \infty}\frac{1}{n}\int_{[0,1)}N_n(x,E)\,dx.
\eeq
The function $N(\cdot)$ is monotone non-decreasing and continuous, and defines a probability measure $dN(E)$ on $\mathbb R$, which is called the {\it density of states measure}. The topological support of $dN$ is equal to the almost sure spectrum of $H(x)$ (see, for example, \cite{CFKS}*{Chapters 9,10}). It is well known that $dN$ is the average of spectral measures of the family $H(x)$:
\beq
\label{ids_spectral}
dN(E)=\int_{[0,1)}\langle d\mathbb{E}_{H(x)}(E)e_0,e_0 \rangle\,dx,
\eeq
where $d\mathbb{E}_{H(x)}$ is the differential of the projection-valued spectral measure of the operator $H(x)$.

Note that, while the existence of the limit \eqref{ids_def} is usually proved by the means of ergodic theorems (see, for example, \cite{CFKS}*{Chapter 9}), our Corollary \ref{cor_pointwise_ids} shows that, for the class of potentials under consideration in Theorem \ref{th_main}, the convergence is actually uniform in $x$:
\beq
\label{eq_uniform_ids}
\frac{1}{n}N_n(x,E)\rightrightarrows N(E),\quad\text{as}\quad n\to +\infty,
\eeq
and one can also use the limit of the left hand side as a definition of $N(E)$.

The hopping term $\Delta$ can move $k$-th eigenvalue of $H_n(x)$ at most by $2$: that is,
$$
|\lambda_k(H_n(x))-\lambda_k(\diag\{f(x+k\omega):k\in [0,n-1]\})|\le \|\Delta\|\le 2.
$$
By passing to the limit, we obtain 
\beq
\label{ids_inv}
|N^{-1}(x)-f(x)|\le 2,\quad x\in (0,1).
\eeq
Note that, since $N$ is constant in spectral gaps of $H$, the inverse function is not uniquely defined for some (at most countably many) values of $x$. Still, \eqref{ids_inv} holds for every choice of the inverse function due to monotonicity of both $N$ and $f$.

The statement of the main Theorem \ref{th_main} requires the definition of the Lyapunov exponent, which is usually defined dynamically using transfer matrices. Unlike \cite{JK,Kach}, our proof will not require any dynamical considerations such as the uniform upper bound or even the ergodic theorems. The only property of the Lyapunov exponent that will be used is the Thouless formula, and therefore, for the purposes of the present paper, we can use it as the definition:
\beq
\label{thouless}
L(E):=\int_{\R} \log|E-E'|\,dN(E').
\eeq
It is known \cite{JK,Kach} that, for $\gamma$-monotone potentials, the density of states measure $dN$ is absolutely continuous with respect to the Lebesgue measure, with density bounded by $\gamma^{-1}$, see also Proposition \ref{prop_monotone_branches}. As a consequence, $\log|E-E'|$ is locally integrable with respect to $dN(E')$ (this also follows from the general theory such as in \cite{CFKS}*{Chapter 9}). Since  \eqref{ids_inv} implies $\log(1+|N^{-1}|)\in \rL^1(0,1)$, the integral in the right hand side of \eqref{thouless} converges, and therefore the Lyapunov exponent is well defined and continuous on $\R$.

\subsection{Irrational rotation and discrepancy}
Let $x_1,\ldots,x_n\in [0,1]$ be real numbers. Following \cite{Kuipers}*{Definition 1.2}, define their {\it discrepancy} by
\beq
\label{eq_discrepancy1}
D_n^{\ast}(x_1,\ldots,x_n)=\sup\limits_{0<t\le 1}\l|\frac{\#\{k\colon x_{k}\in [0,t)\}}{n}-t\r|.
\eeq
Koksma's inequality (see, for example, \cite{Kuipers}*{Theorem 2.5.1}) states that, for any function $f$ on $[0,1]$ of bounded variation, we have
\beq
\label{eq_koksma}
\l|\int_0^1 f(x)\,dx-\frac{1}{n}\sum_{k=1}^n f(x_k)\r|\le \mathrm \|f\|_{\mathrm{BV}([0,1])} D_n^{\ast}(x_1,\ldots,x_n).
\eeq
The following estimate, which follows directly from the triangle inequality, will be useful:
\begin{multline}
\label{eq_discrepancy_sub}
D_{m+n}^{\ast}(x_1,\ldots,x_m,y_1,\ldots,y_n)=\sup\limits_{0<t\le 1}\l|\frac{\#\{k\colon x_{k}\in [0,t)\}+\#\{\ell \colon y_{\ell}\in [0,t)\}}{m+n}-t\r|\\
\le  \frac{m}{m+n}(D_m^{\ast}(x_1,\ldots,x_m)+\frac{n}{m+n}D_n^{\ast}(y_1,\ldots,y_n).
\end{multline}

We will now discuss some relevant and well known results on the irrational rotation. Most of them are well known and contained in, for example, \cite[Section 2.3]{Kuipers}, see also \cite{Khinchine,Raven}. Let $\alpha\in \R\setminus\Q$ have a continued fraction expansion
$$
\alpha=[a_0,a_1,a_2,\ldots]=a_0+\frac{1}{a_1+\frac{1}{a_2+\frac{1}{\ldots}}}.
$$
As usual, denote the continued fraction approximants to $\alpha$ by truncating the above continued fraction:
$$
\frac{p_k}{q_k}:=[a_0,a_1,a_2,\ldots,a_k]=a_0+\frac{1}{a_1+\frac{1}{a_2+\frac{1}{\ldots+\frac{1}{a_k}}}}.
$$
For $y\in \R$, let also
$$
\|y\|:=\dist(y,\Z).
$$
\begin{prop}
\label{prop_rotation_trajectory}
Let $\frac{p_k}{q_k}$ be a continued fraction approximant of $\alpha\in \R\setminus\Q$. Then the corresponding trajectory of irrational rotation has the following structure:
\beq
\label{eq_rotation_trajectory}
\l\{\{j\alpha\}\colon j=0,\ldots,q_k-1\r\}=\l\{\frac{jp_k}{q_k}+\frac{j\theta_j}{q_k q_{k+1}}\colon j=0,\ldots,q_{k}-1\r\},
\eeq
where $|\theta_j|<1$. The points in the left hand side of \eqref{eq_rotation_trajectory} split the interval $(0,1)$ into $q_k$ intervals. The lengths of these intervals can only take two possible values $\|q_{k-1}\alpha\|$ and $\|(q_k-q_{k-1})\alpha\|$. We have
\beq
\label{eq_trajectory_estimates1}
\frac{1}{2 q_k}\le \frac{1}{q_{k}}-\frac{q_{k-1}}{q_k q_{k+1}}\le \|q_{k-1}\alpha\|<\frac{1}{q_k}<\|(q_k-q_{k-1})\alpha\|\le \frac{1}{q_k}+\frac{1}{q_{k+1}}\le \frac{2}{q_k}.
\eeq
As a consequence, one also has
\beq
\label{eq_trajectory_estimates2}
\|j\alpha\|\ge \|q_k\alpha\|, \quad\text{for}\quad 0\le j\le q_{k+1}-1.
\eeq
\end{prop}
Since $j<q_k<q_{k+1}$, \eqref{eq_rotation_trajectory} immediately implies
\beq
\label{eq_discrepancy_q}
D_{q_k}(\{x\},\{x+\alpha\},\ldots,\{x+(q_k-1)\alpha\})\le \frac{1}{q_k}+\frac{1}{q_{k+1}}\le\frac{2}{q_{k}},
\eeq
and therefore, for $n=sq_k+r$, iterating \eqref{eq_discrepancy_sub} yields
\beq
\label{eq_discrepancy4}
D_n^{\ast}(x,\{x+\alpha\},\{x+2\alpha\},\ldots,\{x+(n-1)\alpha\})\le \frac{2}{q_{k}}+\frac{|r|}{n}\le \frac{4(s+|r|)}{n}.
\eeq

\section{Counting functions for box eigenvalues}
Recall that $H_n(x)$ is the restriction of $H(x)$ into $\ell^2[0;n-1]$ with the Dirichlet boundary conditions. We will identify $H_n(x)$ with an $(n\times n)$-block of $H(x)$: for example,
$$
H_5(x)=\begin{pmatrix}
f(x)&1&0&0&0\\
1&f(x+\alpha)&1&0&0\\
0&1&f(x+2\alpha)&1&0\\
0&0&1&f(x+3\alpha)&1\\
0&0&0&1&f(x+4\alpha)
\end{pmatrix}.
$$
The goal of this section is to study the dependence on $x$ of the eigenvalues of $H_n(x)$.

The proof of localization in the Diophantine case, developed in \cite{JK,Kach}, is based on the analysis of $H_q(x)$, where $q$ is a denominator of a continued fraction approximant of $\alpha$. While this analysis allows to extend localization results somewhat beyond the Diophantine case (see \cite{Zhu}), the complete setting of Theorem \ref{th_main} is based on a more advanced argument. The following are two main results of this section.
\begin{enumerate}
	\item For $n=sq+r$, the values of the counting function of the eigenvalues of $H_n(x)$ can change at most by $O(s+|r|)$ as $x$ runs over $[0,1)$ (Corollary \ref{cor_counting}).
	\item The normalized counting function $\frac{1}{n}N_n(x,E)$ is $O\l(\frac{s+|r|}{n}\r)$-close to the integrated density of states $N(E)$ (Corollary \ref{cor_pointwise_ids}).
	\item An analogue of Koksma's inequality holds for functions of bounded variation on $\R$: an average of a function over the eigenvalues of $H_n(x)$ is close to its integral against the density of states measure (Corollary \ref{cor_koksma_ids}).
\end{enumerate}
We also include a streamlined proof of Proposition \ref{prop_constant_counting}, originally obtained in \cite{Kach}, in a somewhat more general setting. The above results are obtained, essentially, by replacing the operator $H_n(x)$ by $s$ copies of $H_q(x)$. As in \cite{JK,Kach}, let
\beq
\label{eq_jump_points_def}
0=\beta_0<\beta_1<\ldots<\beta_{n-1}<\beta_n=1, \quad \{\beta_0,\ldots,\beta_{n-1}\}=\l\{\{-j\alpha\}\colon j\in [0;n-1]\r\}
\eeq
be the points where one of  the diagonal entries of $H_n(x)$ undergoes
a {\it negative} jump discontinuity. These points will be called {\it jump points} (for that purpose, it is convenient to identify $\beta_0$ with $\beta_n$, so that there are exactly $n$ jump points). The intervals $[\beta_j,\beta_{j+1})$ will be called {\it monotonicity intervals}. 

The spectrum of $H_n(x)$ is always simple. Let
$$
-\infty\le E_0(x)<E_1(x)<\ldots<E_{n-1}(x)
$$
be the eigenvalues of $H_n(x)$. By \eqref{eq_f2} each $E_k(x)$ is $\gamma$-monotone on every monotonicity interval $[\beta_j,\beta_{j+1})$:
$$
E_k(y)-E_k(x)\ge \gamma(y-x),\quad \beta_j\le x\le y<\beta_{j+1}.
$$
Thus negative jump discontinuities can only happen as changes from $\beta_j-0$ to $\beta_j$, and, at each of them, the whole matrix $H_n(x)$ undergoes a perturbation whose negative part is rank one. Therefore,
\beq
\label{eq_rank_one}
E_k(x)\ge E_{k-1}(x-0)\quad \forall x\in [0,1),\,\,k=1,2,\ldots,n-1. 
\eeq
The above property allows to change the numeration of the eigenvalues into ``$\gamma$-monotone branches''.
\begin{prop}
\label{prop_monotone_branches}
There exist $\gamma$-monotone functions $\Lambda_k\colon [0,1)\to \R$, that is,
\beq
\label{eq_gamma_monotonicity}
\Lambda_k(y)-\Lambda_k(x)\ge \gamma(y-x),\quad 0\le x\le y<1,
\eeq
such that
\beq
\label{eq_sigma_branches}
\sigma(H(x))=\{\Lambda_0(x),\Lambda_1(x-\beta_{n-1}),\Lambda_2(x-\beta_{n-2})\ldots,\Lambda_{n-1}(x-\beta_1)\},
\eeq
and
$$
E_k(0)=\Lambda_k(-\beta_{n-k}).
$$
\end{prop}
\begin{proof}
It is easy to see that the ordering
$$
\Lambda_k(x):=E_{(j+k)\,\mathrm{mod}\,n}(x+\beta_{n-k}),\quad x\in [\beta_j,\beta_{j+1}).
$$
satisfies the required properties on each interval due to $\gamma$-monotonicity of $E_j$, and at the jump points due to \eqref{eq_rank_one}.
\end{proof}
\begin{remark}
\label{rem_lipschitz_ids}
As mentioned in \cite{Kach}, this immediately implies that the integrated density of states is Lipschitz continuous with the derivative bounded by $\gamma^{-1}$.	
\end{remark}
\begin{remark}
One can also show that the above properties define the functions $\Lambda_k$ uniquely.
\end{remark}
The behavior of each $\Lambda_j$ imitates the behavior of the function $f$, so that the eigenvalue $\Lambda_k(x-\beta_{n-k})$ is $\gamma$-monotone on the same interval as the diagonal entry $f(x-\beta_{n-k})$. In fact, for the diagonal part of $H_n(x)$ (without $\Delta$), one would have $\Lambda_k(x)=f(x)$, with \eqref{eq_sigma_branches} producing all diagonal entries of the operator due to the definition of $\beta_{n-k}$, perhaps not in the original order.
\begin{center}
\includegraphics{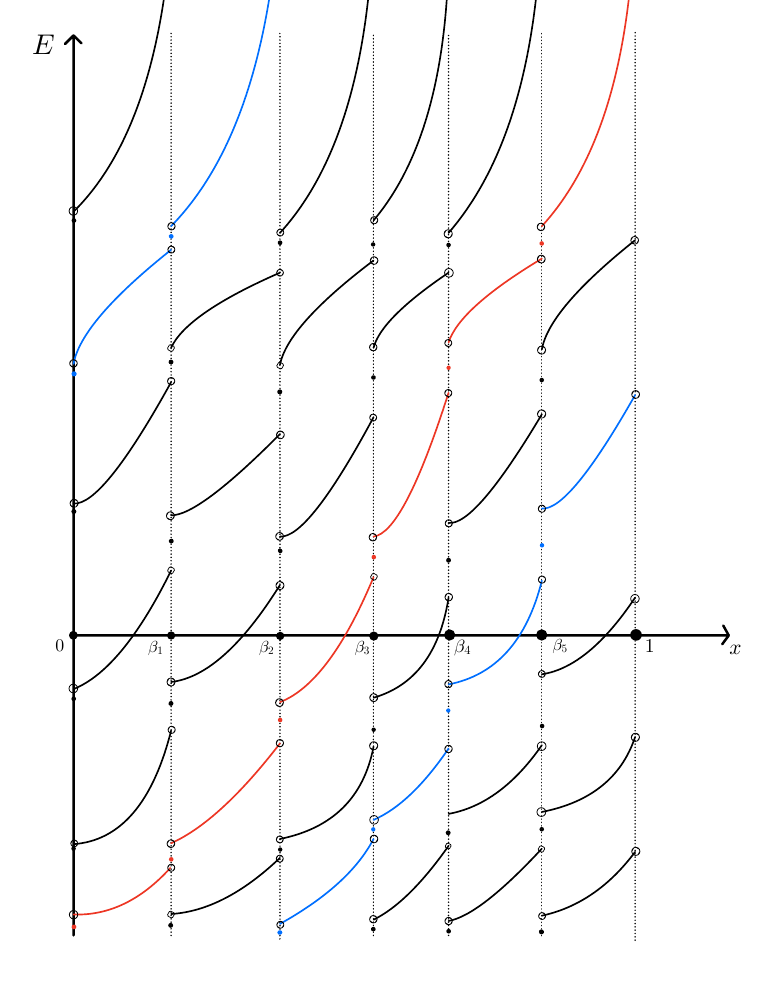}	
\end{center}

The graph below illustrates the expected generic behavior of the eigenvalues of $H_6(x)$, with $f$ bounded from below, unbounded from above, continuous on $(0,1)$ and having a jump discontinuity at the origin: $-\infty<f(0)<f(0+0)$. The red graph is the function $\Lambda_0(x)$, and the blue one is $\Lambda_4(x-\beta_2)$. As one can see, $\Lambda_4(x-\beta_2)$ behaves similarly to $\Lambda_0$ with a translated argument.

As a consequence of Proposition \ref{prop_monotone_branches}, we have the following important bound:
\beq
\label{eq_intersection_points}
\#\{x\in [0,1)\colon E\in \sigma(H_n(x))\}\le n\,\quad \forall E\in \R.
\eeq
The above set can also be described by the equation $\det(H(x)-E)=0$. The estimate \eqref{eq_intersection_points} will play a crucial role in studying the set of $x$ on which $|\det(H(x)-E)|$ is small.

One can now express the eigenvalue counting function as follows:
$$
N_n(x,E)=\#\l(\sigma(H_n(x))\cap [-\infty,E)\r)=\#\{k\colon E_k(x)< E\}=\#\{k\colon \Lambda_k(x+k\alpha)< E\}.
$$
The following result has been obtained in \cite{Kach} for the case of $f$ continuous on $[0,1)$, and only for ``good denominators'' $q_k$ satisfying $q_{k+1}\ge \sqrt{5}q_k$. Among two consecutive denominators, at least one will always be good. For the convenience of the reader, we include the modification of the proof from \cite{Kach} without the above assumptions, potentially with less optimal constants, and in somewhat more streamlined form.
\begin{prop}
\label{prop_constant_counting}
For any denominator $q$ of the continued fraction expansion of any irrational $\alpha$, we have
$$
|N_q(x,E)-N_q(y,E)|\le 16,\quad \forall x,y\in \R.
$$
\end{prop}
\begin{proof}
For $n=q$, Proposition \ref{prop_rotation_trajectory} implies that the length of each monotonicity interval satisfies
$$
\frac{1}{2q}\le \beta_{j+1}-\beta_j\le \frac{2}{q},\quad j=0,\ldots,n-1.
$$
As a consequence, for $x,y\in [0,1)$, the number of jump points between $x$ and $x\pm y$ does not exceed $\lceil 2qy \rceil$. Using \eqref{eq_rank_one}, we have
\beq
\label{eq_counting_monotonicity}
N_q(x+y,E)\le N_q(x,E)+\lceil 2qy \rceil;\quad N_q(x-y,E)\ge N_q(x,E)-\lceil 2qy \rceil,\,\quad \forall y\in [0,1).
\eeq
Similarly to \cite{JK,Kach}, we would like to compare the eigenvalue counting functions of operators $H(x)$ and $H(x+m\alpha)$. Suppose that $0\le m\le q-1$. It is easy to see that
$$
H_q(x+m\alpha)=H_{q-m}(x+m\alpha)\oplus H_m(x+q\alpha)+R,
$$
where $R$ is an operator whose rank does not exceed 2. Note that $\dist(q\alpha,\Z)<1/q$. Depending whether $\{q\alpha\}$ is closer to $1$ or $0$, we can apply the appropriate case of \eqref{eq_counting_monotonicity} and obtain
$$
N_q(x+m\alpha,E)\le N_q(x,E)+4,\quad \{q\alpha\}\in (0,1/2);
$$
$$
N_q(x+m\alpha,E)\ge N_q(x,E)-4,\quad \{q\alpha\}\in (1/2,1).
$$
The above inequalities hold for all $x\in [0,1)$ and $m\in [0;q-1]$. Similarly to \cite{Kach}, after replacing $x$ by $x+m\alpha$ and $m$ by $q-m$ we get, for $m\in [0;q-1]$:
$$
N_q(x+q\alpha)-4\le N_q(x+m\alpha,E)\le N_q(x,E)+4,\quad \{q\alpha\}\in (0,1/q);
$$
$$
N_q(x+q\alpha)-4\le N_q(x+m\alpha,E)\ge N_q(x,E)-4,\quad \{q\alpha\}\in (1-1/q,1).
$$
From these estimates, we can conclude the following: for any two monotonicity intervals $(\beta_{i},\beta_{i+1})$ and $(\beta_j,\beta_{j+1})$, one can pick $x\in (\beta_i,\beta_{i+1})$ and $y\in (\beta_j,\beta_{j+1})$ such that $|N_q(x,E)-N_q(y,E)|\le 4$. Now, assume that  $(\beta_{i-1},\beta_{i})$, $(\beta_{i},\beta_{i+1})$, $(\beta_{i+1},\beta_{i+2})$ are three consecutive monotonicity intervals, and the points $y_1$, $y_3$ are chosen using the above observation:
$$
|N(y_1,E)-N(y_3,E)|\le 4,\quad y_1\in (\beta_{i-1},\beta_{i}),\quad y_3\in (\beta_{i+1},\beta_{i+2}).
$$
Then
\begin{multline*}
N(y_1,E)+1\ge N(\beta_{i}-0,E)+1\ge N(\beta_i+0,E)\ge N(\beta_{i+1}-0,E)\\ \ge N(\beta_{i+1}+0,E)-1\ge N(y_3,E)-1.
	\end{multline*}
Therefore, we have
$$
|N(\beta_{i+1}-0,E)-N(\beta_i+0,E)|\le 6.
$$
The argument also applies for $\beta_i=0$ or $\beta_{i+1}=1$, if one uses the convention $\beta_{j+q}=\beta_j$.
We can conclude that the variation of $N(\cdot,E)$ on each monotonicity interval does not exceed $6$. 
In view of previous considerations, this proves the main claim.
\end{proof}
We are now ready to discuss the results for arbitrary values of $n$. Note that the numerical values of various constants are not of particular importance, and we sacrificed optimizing some of them for the sake of a more direct argument.
\begin{cor}
\label{cor_counting}
Let $n=sq+r$, where $q$ is a denominator of the continued fraction expansion of $\alpha$. Then
\beq
\label{eq_counting_q1}
|N_n(x,E)-s N_q(y,E)|\le 16s+4(s+|r|),\quad \forall x,y\in \R.
\eeq
As a consequence,
\beq
\label{eq_counting_q2}
|N_n(x,E)-N_n(y,E)|\le 64s+16(s+|r|),\quad \forall x,y\in \R.
\eeq
\end{cor}
\begin{proof}
A perturbation of rank at most $2(s+|r|)$ will transform $H_n(x)$ into 
$$
H_q(x)\oplus H_q(x+\alpha)\oplus\ldots\oplus H_q(x+(s-1)\alpha).
$$
From Proposition \ref{prop_constant_counting} applied to each summand, we have
\beq
\label{eq_counting_q}
|N_n(x,E)-sN_q(y,E)|\le 2(s+|r|)+16s,\quad \forall x,y\in \R.
\eeq
The estimate \eqref{eq_counting_q2} now follows from the triangle inequality.
\end{proof}
\begin{cor}
\label{cor_pointwise_ids}
Let $n=sq+r$, where $q$ is a denominator of the continued fraction expansion of $\alpha$. Then
\beq
\label{eq_pointwise_ids1}
\l|\frac{1}{n}N_{n}(x,E)-\frac{1}{q}N_q(y,E)\r|\le \frac{4|r|}{sq}+\frac{20}{q},\quad \forall x,y\in [0,1).
\eeq
and similarly for $n_1=s_1 q+r_1$, $n_2=s_2 q+r_2$
\beq
\label{eq_pointwise_ids3}
\l|\frac{1}{n_1}N_{n_1}(x,E)-\frac{1}{n_2}N_{n_2}(y,E)\r|\le \frac{4|r_1|}{s_1q}+\frac{4|r_2|}{s_2 q}+\frac{40}{q}.
\eeq
As a consequence, $\frac{1}{n}N_n(x,E)\rightrightarrows N(E)$ uniformly in $x$, and one has
\beq
\label{eq_pointwise_ids2}
\l|\frac{1}{n}N_{n}(x,E)-N(E)\r|\le \frac{4|r|}{sq}+\frac{40}{q}.
\eeq
\end{cor}
\begin{proof}
The first two estimates directly follow from Corollary \ref{cor_counting} and the triangle inequality, and the rest follows from the standard uniform Cauchy sequence argument and the definition \eqref{ids_def}. Note that this argument also implies directly that the limit in \eqref{ids_def} exists for the class of operators satisfying the conclusion of Proposition \ref{prop_constant_counting}.
\end{proof}
The last inequality \eqref{eq_pointwise_ids2} can also interpreted in terms of discrepancy \eqref{eq_discrepancy1} and also follows directly from the definition of the IDS:
\beq
\label{eq_discrepancy_ids}
D_n^*\l(N(E_0(x)),N(E_1(x)),\ldots,N(E_{n-1}(x))\r)\le \frac{4|r|}{sq}+\frac{40}{q}.
\eeq
As a consequence, one can obtain the following analogue of Koksma's inequality:
\begin{cor}
\label{cor_koksma_ids}
Let $g\colon \R\to \R$ be a function of bounded variation. Then
\beq
\label{eq_koksma_ids}
\l|\int_{\R} g(E')dN(E')-\frac{1}{n}\sum_{j=0}^{n-1} g(E_k(x))\r|\le \|g\|_{\mathrm{BV}(\R)}\l(\frac{4|r|}{sq}+\frac{40}{q}\r).
\eeq
\end{cor}
\begin{proof}
Follows from the relation
$$
\int_{\R} g(E')dN(E')=\int_0^1 g(N^{-1}(s))\,ds
$$
and Koksma's inequality \eqref{eq_koksma} applied to the function $x\mapsto g(N^{-1}(x))$, in view of \eqref{eq_discrepancy_ids}. Note that $N^{-1}(x)$ is well defined except maybe for a countable subset of values of $x$ which do not affect the integral.
\end{proof}
\section{Large deviations for box eigenvalues}
In the notation of the previous section, let
$$
P_n(x,E):=\det(H_n(x)-E)=\prod_{k=0}^n (E_k(x)-E)=\prod_{k=0}^n (\Lambda_k(x-\beta_{n-k})-E).
$$
The analysis of these determinants is crucial for non-perturbative one-dimensional localization problems, both for monotone potentials \cite{JK,Kach} and in general \cite{Lana94,Lana99,JHY}. We have
$$
\frac{1}{n}\log|P_n(x,E)|=\frac{1}{n}\l(\log|E-E_0(x)|+\ldots+\log|E-E_{n-1}(x)|\r).
$$
It would be natural to apply Corollary \ref{cor_koksma_ids} to the function $E'\mapsto\log|E-E'|$, and obtain, in view of the Thouless formula,
\beq
\label{eq_no_ldt}
\frac{1}{n}\log|P_n(x,E)|=\int_{\R} \log|E-E'|\,dN(E')+o(1)=L(E)+o(1),\quad n\to +\infty,
\eeq
after a suitable choice of $r,s,q$ (which is always possible). However, the above function is not of bounded variation. This leads to existence of some values of $x$ and $j$ for which \eqref{eq_no_ldt} is violated, because $\log|E-E_k(x)|$ is close to $+\infty$ or $-\infty$.

In order to quantify the above situation, let us choose some cut-off parameters 
$$
0<B_-<1<10<B_+<+\infty
$$ 
and define
\beq
\label{eq_def_p_minus}
P_n^{<B_-}(x,E):=\prod\limits_{k\colon |\Lambda_k(x-\beta_{n-k})-E|< B_-}(\Lambda_k(x-\beta_{n-k})-E),
\eeq
\beq
\label{eq_def_p_plus}
P_n^{>B_+}(x,E):=\prod\limits_{j\colon |\Lambda_k(x-\beta_{n-k})-E|> B_+}(\Lambda_k(x-\beta_{n-k})-E),
\eeq
\beq
\label{eq_def_p_mid}
P_n^{\mathrm{mid}}(x,E)(x,E):=\prod\limits_{k\colon B_-\le |\Lambda_k(x-\beta_{n-k})-E|\le B_+}(\Lambda_k(x-\beta_{n-k})-E).
\eeq
We have
\beq
\label{eq_pn_product}
P_n(x,E)=P_n^{<B-}(x,E)P_n^{\mathrm{mid}}(x,E)P_n^{>B+}(x,E).
\eeq
Our goal will be to show that $P_n^{\mathrm{mid}}(x,E)$, essentially, satisfies \eqref{eq_no_ldt}, and the possible deviations can only be caused by $P_n^{<B_-}$ being very small or $P_n^{>B_+}$ being very large. In both cases, we obtain two-sided bounds on $P_n^{<B_-}(x,E)$ and $P_n^{>B_+}(x,E)$ that allow us to control the corresponding exceptional sets.

\subsection{Very small eigenvalues} Let $\Lambda_k(\cdot-\beta_{n-k})$ be one of the monotone eigenvalue branches defined in Proposition \ref{prop_monotone_branches}, and $E\in \R$. Let 
\begin{multline}
\label{eq_intersection_point}
z_j=z_j(E):=\inf\{x\in(0,1)\colon \Lambda_k(x-\beta_{n-k})\ge E\}\\=\sup\{x\in(0,1)\colon \Lambda_k(x-\beta_{n-k})\le E\}\in [0,1],
\end{multline}
where we assume that $\inf$ over an empty set is $1$ and $\sup$ over an empty set is $0$. In other words, $z_k$ is the (unique) pre-image of $E$ under $\Lambda_k(\cdot-\beta_{n-k})$, or the closest possible point of $[0,1]$ to it (note that there will be no difference between $z_k=0$ and $z_k=1$). From $\gamma$-monotonicity of $\Lambda_k$, one can see that
\beq
\label{eq_distances_lowerbound}
|\Lambda_k(x-\beta_{n-k})-E|\ge \gamma|\{x-\beta_{n-k}\}-z_k+\beta_{n-k}|\ge \gamma\|x-z_k\|=\gamma \dist(x-z_k,\Z).
\eeq
Using the definition of $P_n^-(x,E)$ \eqref{eq_def_p_minus}, we have 
\beq
\label{eq_p_minus_twosided}
1\ge |P_n^{<B_-}(x,E)|=\prod\limits_{k\colon |\Lambda_k(x-\beta_{n-k})-E|< B_- }\left\{(\Lambda_k(x-\beta_{n-k})-E)\right\}\ge  \prod_{k\colon  \gamma\|x-z_k\|<B_-} \gamma\|x-z_k\|.
\eeq
Note that the right hand side of the last inequality can have more factors than the left hand side. However, the choice of $B_-$ guarantees that none of the factors would exceed $1$, thus preserving the inequality.

The points $\{z_k(E)\}$ will be called {\it intersection points}. In order to obtain meaningful estimates of $P_n^{<B_-}$, one needs some information on their distribution.
\begin{lemma}
\label{lemma_quantity_intersections}
Let $n=sq+r$, where $q$ is a denominator of the continued fraction expansion of $\alpha$, and $0\le x<y<1$. Then
\begin{multline*}
N_n(y,E)-N_n(x,E)\le \\
\le\#\{\text{jump points between $x$ and $y$}\}-\#\{\text{intersection points between $x$ and $y$}\}.
\end{multline*}
As a consequence,
$$
\#\{\text{intersection points between $x$ and $y$}\}\le s\lceil q(y-x)+1\rceil+64s+16(s+|r|).
$$
In particular, the number of factors in the right hand side of $\eqref{eq_p_minus_twosided}$ satisfies
\beq
\label{eq_p_minus_factors}
\#\{j\in [0;n-1]\colon  \gamma\|x-z_k\|<B_-\}\le \frac{2s q B_-}{\gamma}+100(s+|r|)+20.
\eeq
\end{lemma}
\begin{proof}
The counting function will decrease by $1$ if one passes an intersection point, and can increase at most by $1$ at each jump point, which implies the first claim.
The second claim then follows from Corollary \ref{cor_counting}. Note that the same point can be a jump point and/or an intersection point for multiple eigenvalue branches. If that happens, it needs to be counted multiple times. The last claim follows from the fact that each intersection point appearing in the right hand side of \eqref{eq_p_minus_twosided} is contained in a $B_-(n)/\gamma$-neighborhood of $x$ (modulo $1$).
\end{proof}
\subsection{Very large eigenvalues}Recall that, by the usual perturbation theory, we have
$$
|\lambda_k(H_n(x))-\lambda_k(\diag\{f(x+k\omega):k\in [0;n-1]\})|\le \|\Delta\|\le 2,
$$
where $\lambda_k$ denotes $k$th largest eigenvalue counting multiplicity. As a consequence, we have the following two-sided bound:
\begin{multline}
\label{eq_p_plus_twosided_twosided}
\prod_{k:|f(x+k\alpha)-E|>B_++2}\l(|f(x+k\alpha)-E|-2\r)\le 
P_{n}^{>B_+}(x,E)\le \\ \le \prod_{k:|f(x+k\alpha)-E|>B_+-2}\l(|f(x+k\alpha)-E|+2\r).
\end{multline}
Moreover, the same inequalities also hold for the numbers of factors in all three expressions. Let us estimate the number of factors in the right hand side of \eqref{eq_p_minus_twosided} for $n=rq+s$, where $q$ is a denominator of the continued fraction expansion of $\alpha$. Using \eqref{eq_discrepancy4}, Koksma's inequality, and Markov's inequality, we have
\begin{multline}
\label{eq_p_plus_factors}
\#\{k:|f(x+k\alpha)-E|>B_+-2\}\\ \le 4(s+|r|)+n|\{x\in [0,1)\colon \log\l(|f(x)-E|+2\r)>\log B_+\}|\\
\le 4(s+|r|)+\frac{n}{\log B_+}\int_0^1\log\l(|f(x)-E|+2\r)\,dx.
\end{multline}
\subsection{The remaining eigenvalues and the large deviation theorem} Define the following truncated absolute value function:
$$
|E-E'|_{[B_-,B_+]}:=\begin{cases}
1,&|E-E'|< B_-;\\
|E-E'|,&B_-\le |E-E'|\le B_+;\\
1,&|E-E'|> B_+.
\end{cases}
$$
Clearly, the corresponding truncated version of the logarithm $E'\mapsto \log\l(|E-E'|_{[B_1,B_2]}\r)$ is of bounded variation:
\beq
\label{eq_g_variation}
\|\log |E-\cdot|_{[B_-,B_+]}\|_{\mathrm{BV}(\R)}=4\log (B_+/B_-).
\eeq
We have
$$
|P^{\mathrm{mid}}_n(x,E)|=\prod_{j=0}^{n-1}\l|E-E_k(x)\r|_{[B_-,B_+]}.
$$
From the analogue of Koksma's inequality stated in Corollary \ref{cor_koksma_ids}, we have that
\begin{multline}
\label{eq_q_lyapunov}
\l|\frac{1}{n}\log |P^{\mathrm{mid}}_n(x,E)|-\int \log |E-E'|_{[B_-,B_+]}dN(E')\r|\\ \le \l(\frac{8|r|}{sq}+\frac{128}{q}\r)\log (B_+/B_-).
\end{multline}
The estimates obtained above can be summarized in the following large deviation theorem.
\begin{theorem}
\label{th_ldt}
Let $n=sq+r$, $|r|<q$, where $q$ is a denominator of a continued fraction approximant of $\alpha$. Let $0<B_-<1<10<B_+<+\infty$. We have
$$
P_n(x,E)=P_n^{<B_-}(x,E)P_n^{\mathrm{mid}}(x,E)P_n^{>B_+}(x,E),
$$
where
\beq
\l|\frac{1}{n}\log P_n^{\mathrm{mid}}(x,E)-\int_{\R} \log |E-E'|_{[B_-,B_+]}dN(E')\r|\\ \le 300\frac{s+|r|}{n}\log(B_+/B_-),
\eeq
and $P_n^{+}(x,E)$ and $P_n^{-}(x,E)$ satisfy the two-sided bounds \eqref{eq_p_minus_twosided}, \eqref{eq_p_plus_twosided_twosided} with the numbers of factors bounded by \eqref{eq_p_minus_factors}, \eqref{eq_p_plus_factors}, respectively. The estimates are uniform on any compact energy interval.
 \end{theorem}
\begin{remark}
\label{rem_large_deviation_name}
Usually, the name ``large deviation theorem'' is used for a result that describes the set of parameters where the value of a certain object (in our case, $P_n(x,E)$) is exponentially far away from it's expected value (in other words, \eqref{eq_no_ldt} is violated). The above results is an opposite statement for $P_n^{\mathrm{mid}}(x,E)$. However, its conclusion can also be interpreted as that all large deviation events are caused either by $P_n^{<B_-}$ being exponentially small, or $P_n^{>B_+}$ being exponentially large.
\end{remark}

It will be convenient to introduce
\beq
\label{eq_def_l_corr}
L_{\mathrm{corr}}(E,B_-,B_+):=\l|L(E)-\int_{\R} \log |E-E'|_{[B_-,B_+]}dN(E')\r|,
\eeq
that is, the difference between the Lyapunov exponent and the ``partial Lyapunov exponent'' appearing in Theorem \ref{th_ldt}. Clearly, from the Thouless formula, one has
$$
L_{\mathrm{corr}}(E,B_-,B_+)\to 0\quad \text{as}\quad B_-\to +0,\,B_+\to +\infty
$$
uniformly in $E$ on any compact energy interval. 

In order to obtain meaningful bounds in Theorem \ref{th_ldt}, it is not always beneficial to choose the largest possible value of $q$, since one may not have $r=o(n)$ if $s$ is too small (the worst possible case is, say, $n\sim 3q/2$). It will be convenient to fix a specific more balanced choice and restate Theorem \ref{th_ldt} in a form that demonstrates improvement of the bounds as $n\to +\infty$.
 \begin{cor}
 \label{cor_ldt_specific_s}
 Under the assumptions of Theorem $\ref{th_ldt}$, let $q(n)$ be the largest denominator among the continued fraction approximants of $\alpha$ with the property
\beq
\label{eq_q_n_def}
n=sq(n)+r(n),\quad 0\le |r(n)|\le \sqrt{n}
\eeq
and, as a consequence,
\beq
\label{eq_q_n_def_consequence}
\frac{s+|r|}{n}\le \frac{1}{q(n)}+\frac{|r(n)|}{n}\le \frac{1}{q(n)}+\frac{1}{\sqrt{n}}.
\eeq
Then, one has
$$
\l|\frac{1}{n}\log P_n^{\mathrm{mid}}(x,E)-L(E)\r|\\ \le \frac{1}{n}L_{\mathrm{corr}}(E,B_-,B_+)+300\l(\frac{1}{q(n)}+\frac{|r(n)|}{n}\r)\log(B_+/B_-).
$$
\end{cor}
Finally, it will be useful to have a uniform upper bound on the partial determinants that include $P_n^{<B_-}(x,E)$, which follows immediately from above combined with \eqref{eq_p_minus_twosided}. A similar bound also follows from a uniform bound on transfer matrices for rough cocycles obtained in \cite{LanaMavi}, as previously used in \cite{JK,Kach}. Here, the bound is more precise and explicit.
\begin{cor}
\label{cor_uniform_upper}
In the notation of Corollary $\ref{cor_ldt_specific_s}$, one has
\begin{multline}
\label{eq_uniform_upper}
|P_n^{<B_-}(x,E)P_n^{\mathrm{mid}}(x,E)|\le \\ \le \exp\l\{n L(E)\l(1+L_{\mathrm{corr}}(E,B_-,B_+)+300\l(\frac{1}{q(n)}+\frac{|r(n)|}{n}\r) \log(B_+/B_-)\r)\r\},
\end{multline}
where the estimate is uniform on any compact energy interval.
\end{cor}
\begin{remark}
If $f$ is bounded, one can always choose $B_+\ge \max_x|f(x)|+2$ which will make $P_n^+(x,E)=1$, thus obtaining a bound on the complete determinant. In general, as in \cite{Kach}, a uniform upper bound is not possible and one has to treat $P_n^{>B_+}(x,E)$ separately.
\end{remark}
\section{Proof of localization: deterministic estimates and statements of the main lemmas}



The proof of localization will be based on the Green's function method. The arguments of such kind involve covering the integer line $\Z$ by a collection of intervals of various sizes, such that the Green's functions of the restrictions of $H(x)$ into each interval satisfy a certain kind of off-diagonal decay (``good intervals''). The particular version of this argument, originally appeared in \cite{Lana94} and developed in \cite{Lana99} and later works, uses Cramer's rule and large deviations theorems in order to find a good interval among translations of an interval of a fixed length containing a given point.

In arithmetic localization results for quasiperiodic operators with $\beta(\alpha)>0$, the choice of good intervals is usually based on distinguishing between resonant and non-resonant points of $\Z$, see \cite{JL2,JL3,JHY}. Once the good intervals are found, the rest of the proof is based on a somewhat standard ``patching argument'', essentially relating the eigenfunction decay on smaller scales to that on larger scales (either by using the resolvent identity, or by iterating the Poisson formula).

In this section, we prove Theorem \ref{th_main} modulo several
Lemmas. in Subsection \ref{subsection_patching} we gather the arguments of a more general kind that
do not directly involve quasiperiodic structure of the
operator. Specifically, we
define Green's functions, good intervals, and regular points, as well
as establish two convenient ``patching lemmas'' that should be useful to
significantly stremline the treatment of resonances also in other
situations.

In Subsection
\ref{subsection_lemmas}, we state the main lemmas that will be used in
the proof of the main result. Each lemma will, essentially, state that
a certain class of points, depending on $\alpha$, has certain degree
of regularity. Those Lemma will be proved in Section \ref{proofs}. In Subsection \ref{subsection_proof_main}, we establish
the main result, Theorem \ref{th_main}, by checking that good
intervals provided by the above lemmas can be patched together, using
the considerations from Subsection \ref{subsection_patching}. In
Subsection \ref{subsection_improvements}, we discuss some improvements
in the exponential decay of eigenfunctions which bring the bounds
closer to those known for some specific models \cite{JL3,JHY}. 

\subsection{Green's functions, good intervals, and paths of minimum weight}
\label{subsection_patching}
We will start from some general definitions and lemmas, with $H$ being a general self-adjoint one-dimensional Schr\"odinger operator on $\ell^2(\Z)$.

Fix $\sigma>0$. We will say that $m\in \Z$ is $(\mu,n)$-{\it regular} if there exists an integer interval $[n_1;n_2]$ with 
\beq
\label{eq_length_regular_interval}
m\in [n_1;n_2],\quad n_2=n_1+n-1,\quad |m-n_j|\ge \sigma n=\sigma|n_2-n_1+1|, ,\quad j=1,2,
\eeq
and
\beq
\label{eq_green_good1}
|(H_{[n_1,n_2]}-E)^{-1}(n_j,m)|\le e^{-\mu|m-n_j|},\quad j=1,2.
\eeq
The corresponding interval $[n_1;n_2]$ will be referred to as a {\it good interval around $m$}. In general, an interval $[n_1,n_2]$ is {\it good} if it satisfies \eqref{eq_green_good1}  for every $m$ within the range defined by \eqref{eq_length_regular_interval}.

In the proof of localization, one starts with a generalized eigenfunction $\psi=\{\psi(n)\}_{n\in \Z}$ satisfying the eigenvalue equation normalized as follows:
\beq
\label{eq_initial_condition}
H\psi=E\psi,\quad |\psi(0)|^2+|\psi(1)|^2=1,\quad |\psi(m)|\le C_{\psi}(1+|m|).
\eeq
The central estimate relating the decay of $\psi$ with off-diagonal decay of the Green's functions follows from the Poisson formula: if $H(x)\psi=E\psi$ and $m\in [m_1,m_2]$, then
\beq
\label{eq_poisson}
\psi(m)=-(H_{[m_1,m_2]}-E)^{-1}(m_1,m)\psi(m_1-1)-(H_{[m_1,m_2]}-E)^{-1}(m_2,m)\psi(m_2+1).
\eeq
It will be convenient to introduce a somewhat more flexible $\psi$-specific notion of regularity. We will say that $m\in\Z$ is {\it $(\psi,\mu,n)$-regular} if there exists an interval satisfying \eqref{eq_length_regular_interval} such that
\beq
\label{eq_regularity_psi}
|\psi(m)|\le e^{-\mu |m-n_1|}|\psi(n_1-1)|+e^{-\mu |m-n_2|}|\psi(n_2+1)|.
\eeq
The corresponding interval $[n_1,n_2]$ will be called a {\it $\psi$-good} interval (around $m$). Clearly, $(\mu,m)$-regularity implies $(\psi,\mu,m)$-regularity.

Let $\psi$ be a solution of the eigenvalue equation $H\psi=E\psi$, and $\mu,L>0$. We will say that a pair of integer intervals 
$$
[N_1,N_2]\subset [N_1^-,N_2^+],\quad  N_1^-<N_1<N_2<N_2^+
$$
{\it admits a $(\psi,\mu,L)$-good collection of intervals} if every point $m\in [N_1,N_2]$ is $(\psi,\mu_m,L_m)$-regular, and the corresponding $\psi$-good intervals are contained in $[N_1^-,N_2^+]$:
$$
m\in [n_1,n_2]\subset [N_1^-,N_2^+],\quad |n_2-n_1+1|=L_m,\quad |m-n_j|\ge \sigma|n_2-n_1|,
$$
where
\beq
\label{eq_good_lower_bound}
\mu_m\ge \mu>0,\quad L_m\ge L>0, \quad e^{\mu \sigma L}>2.
\eeq
\begin{remark}
The parameters $\mu,L$ serve as lower bounds on the actual values $L_m,\mu_m$. In general, larger exponents $\mu_m$ are always better. Larger lengths $L_m$ are better in most cases, unless there are specific requirements on the location of the endpoints of the interval where better bounds on $\psi(n_1-1)$ and $\psi(n_2+1)$ are available for applying \eqref{eq_poisson}.
\end{remark}

Let $m\in [N_1,N_2]$. A {\it terminating path} (associated to a $(\psi,\mu,L)$-good collection of intervals defined above) is a finite sequence of integer points $P=(m_0,m_1,m_2,\ldots,m_p)$ with the following properties:
\begin{enumerate}
	\item $n_0=m$.
	\item $m_{j+1}=n_1-1$ or $m_{j+1}=n_2+1$, where $[n_1,n_2]$ is a $\psi$-good interval around $m$.
	\item $m_p\in [N_1^-,N_2^+]\setminus[N_1,N_2]$
	\item $\{m_0,\ldots,m_{p-1}\}\subset [N_1,N_2]$.
\end{enumerate}
For $P=(m_0,m_1,m_2,\ldots,m_p)$, let also
$$
\ell(P):=p;\quad |P|:=\sum_{j=0}^{p-1} |m_{j+1}-m_j|;\quad w(P):=\sum_{j=0}^{p-1} \mu_j(|m_{j+1}-m_j|-1|.
$$
The quantity $w(P)$ will be referred to as the {\it weight} of the path $P$. The following two propositions are adaptations of some widely used results in Anderson localization and can be stated either in the form of iterating the Poisson's formula, or the resolvent identity. It is clear that either iteration leads to a formula expressing $\psi(m)$ through a summation over some collection of paths, where the length of a path represents the depth of iterations. Under some general assumptions, the following is true:
\begin{enumerate}
	\item If one considers paths of arbitrary length, it is sufficient to consider only terminating paths (Lemma \ref{lemma_terminating}).
	\item Modulo a correction that can be absorbed into $o(1)$ in the exponent, it is sufficient to consider the contribution of the terminating paths of smallest weight (Lemma \ref{lemma_dominating}).
\end{enumerate}
\begin{lemma}
\label{lemma_terminating}
Suppose that a pair $[N_1,N_2]\subset [N_1^-,N_2^+]$ admits a $(\psi,\mu,L)$-collection of intervals. Then
\beq
\label{eq_terminating}
|\psi(m)|\le \sum_{P}e^{-w(P)}|\psi(m_p)|,
\eeq
where the sum is considered over all terminating paths starting at $m$, and $m_p$ is the last point of the terminating path $P$ under consideration.
\end{lemma}
\begin{proof}
Regularity of $\psi$ at $m$ implies
$$
|\psi(m)|\le e^{-\mu_m|m-n_1|}|\psi(n_1-1)|+e^{-\mu_m|m-n_2|}|\psi(n_2+1)|.
$$
We will keep iterating the above estimate with $m$ replaced by $n_1-1$ and $n_2+1$, as long as the new points are inside $[N_1,N_2]$, and stop the iterations once they reach $[N_1^-,N_2^+]\setminus[N_1,N_2]$. After $k$ iterations, we will have the sum over all terminating path $P$ with $|P|\le k$, plus the sum over all non-terminating paths of length $k$. Since each exponential factor is at most $e^{-\mu\sigma L}$, the total contribution of the non-terminating paths of length $k$ can be bounded by
$$
M 2^k e^{-\mu \sigma kL},\quad \text{where}\,\,\,M=\max\{|\psi(m)|\colon m\in [N_1^-,N_2^+]\}.
$$ Since $e^{\mu\sigma L}>2$, the latter goes to zero as $k\to +\infty$, and the iterated estimate becomes \eqref{eq_terminating}.
\end{proof}
In the above notation, fix $m_0\in [N_1,N_2]$. Denote by $\mathcal P_+$ and $\mathcal P_-$ the set of all paths starting at $m_0$ and terminating on $(N_2,N_2^+]$ and $[N_1^-,N_1)$, respectively. Let
$$
a:=\frac{16(1+\mu)\log(\sigma L)}{\mu\sigma L}.
$$
\begin{lemma}
\label{lemma_dominating}
Under the assumptions of Lemma $\ref{lemma_terminating}$, we have
$$
\sum_{P\in \mathcal P_{\eta}}\exp\l(-w(P)\r)\le \exp\l(-(1-a)\min\{w(P)\colon P\in \mathcal P_{\eta}\}\r),\quad \eta\in \{+,-\}.
$$
\end{lemma}
\begin{proof}Fix a point $m_0\in [N_1+,N_2^-]$, and consider all paths $P$ starting at $m_0$ with $\ell(P)=\ell$, $|P|=N$. Any such path is completely determined by the numbers $\{m_{j+1}-m_j\colon j=0,\ldots,\ell-1\}$. Since $\sum |m_{j+1}-m_j|=N$, there are at most $\binom{N}{\ell}$ to choose the absolute values $|m_{j+1}-m_j|$, and at most $2^{\ell}$ ways to choose the signs. As a consequence, there are at most $2^{\ell}\binom{N}{\ell}$ such paths. We have
$$
\binom{N}{\ell}\le \frac{N^{\ell}}{\ell!}\le e^{\ell}e^{\ell\log N-\ell\log \ell}=e^{\ell}e^{\ell\log (N/\ell)}.
$$
The function $t\mapsto e^{t\log N/t}$ is non-decreasing on $\l(0,\frac{N}{e}\r]$. Since $\ell\le \frac{N}{\sigma L}\ll \frac{N}{e}$, we have (after absorbing the factor $(2e)^{\ell}$ into an extra factor of $2$ in the exponent)
$$
2^{\ell}\binom{N}{\ell}\le e^{2N\frac{\log(\sigma L)}{\sigma L}}.
$$
By summing over all possible values of $\ell$, we have that
\beq
\label{eq_path_A}
\#\{P\colon |P|=N\}\le \sum_{\ell=0}^{\lceil\frac{N}{\sigma L}\rceil}2^{\ell}\binom{N}{\ell}\le  e^{3N\frac{\log(\sigma L)}{\sigma L}},\quad \#\{P\colon |P|\le N\}\le  e^{4N\frac{\log(\sigma L)}{\sigma L}}.
\eeq
Now, assume that $P$ is a path starting at $m_0$, terminating on $[N_1^-;N_1)$, and with $w(P)=k$. It is easy to see that one must have $|P|\le \frac{2 k}{\mu}$. Using \eqref{eq_path_A}, we can estimate the number of paths with $w(P)=k$ (in fact, with $w(P)\le k$) by
\begin{multline*}
\frac{2k}{\mu}\exp\l\{k\frac{8\log(\sigma L)}{\mu\sigma L}\r\}\le \exp\l\{k\left(\frac{8 \log(\sigma L)}{\mu\sigma L}+\frac{2\log k}{k}\right)\r\}\\ \le \exp\l\{k\frac{8(1+\mu)\log(\sigma L)}{\mu\sigma L}\r\}=e^{\frac{ka}{2}},	
\end{multline*}
where we assumed that $L$ is large enough depending on $\mu$ and $k\ge \sigma L$ (in other words, $P$ has at least one non-trivial segment).

We can now get back to the original estimate. Let
$$
k_{\pm}:=\min\{w(P)\colon P\in \mathcal P_{\pm}\}>0.
$$
Assume that $L$ is large enough so that $a\le 1/100$. Then
$$
\sum_{P\in \mathcal P{\pm}}\exp\{-w(P)\}\le \sum_{k=k_{\pm}}^{+\infty}e^{-k(1-a/2)}=e^{-k_{\pm}(1-a/2)}\frac{1}{1-e^{-(1-a/2)}}\le 2 e^{-k_{\pm}(1-a/2)}\le e^{-k_{\pm}(1-a)}.
$$
Here we used again the fact that $k_{\pm}\ge \sigma L$ and $L$ is large enough.
\end{proof}

\begin{remark}
\label{rem_same_exponent}
Under the assumptions of Lemmas \ref{lemma_terminating} and \ref{lemma_dominating}, let
$$
|\psi(N_1')|:=\max\{|\psi(m)|\colon m\in (N_2;N_2^+]\};
$$
$$
|\psi(N_2')|:=\max\{|\psi(m)|\colon m\in [N_1^-;N_1)\}.
$$
Let also
$$
w(P_1):=\min\{w(P)\colon P=(m_0,\ldots,m_p),\,m_0=m;\,\,m_p\in [N_1^-;N_1)\};
$$
$$
w(P_2):=\min\{w(P)\colon P=(m_0,\ldots,m_p),\,m_0=m;\,\,m_p\in (N_2;N_2^+]\}
$$
be the smallest possible weights of paths terminating on $[N_1^-;N_1)$, $(N_2;N_2^+]$, respectively. Then, by Lemmas \ref{lemma_terminating}, \ref{lemma_dominating},
\beq
\label{eq_poisson_applied}
|\psi(m)|\le e^{-w(P_1)(1-o(1))}|\psi(N_1')|+e^{-w(P_2)(1-o(1))}|\psi(N_2')|.
\eeq
In particular, one also has
\beq
\label{eq_poisson_applied2}
|\psi(m)|\le e^{-\mu(1-o(1))|m-N_1|}|\psi(N_1')|+e^{-\mu(1-o(1))|m-N_2|}|\psi(N_2')|.
\eeq
Here, $o(1)$ is as $L\to \infty$. Note that $w(P)$ can be slightly smaller than $\mu |P|$, due to the fact that there is a loss of $1$ in the path length each iteration. However, this loss can be absorbed into $o(1)$. Also, as stated, the paths $P_1$, $P_2$ do not necessarily terminate exactly at $N_1'$, $N_2'$.

It may not necessarily be possible to calculate exact values of $w(P_1)$ and $w(P_2)$, in which case one would have to replace them by some lower bounds. Note that a stronger lower bound on $w(P)$ yields a stronger upper bound on $|\psi(m)|$. An estimate \eqref{eq_poisson_applied2} is an example of such bound that is always available. In some cases, such as the case of the Liouville transition, the position and size of good intervals will allow us to obtain better bounds.
\end{remark}
The following is also a standard consequence of the Poisson's formula.
\begin{prop}
\label{prop_origin}
Suppose that $\psi$ is a generalized eigenfunction of $H$ satisfying $\eqref{eq_initial_condition}$. Let $\mu>0, \sigma\in (0,1/2)$. For $r\ge r_0=r_0(\sigma,C_{\psi},\mu)$, the origin $0\in \Z$ cannot be $(\psi,\mu,r)$-regular.
\end{prop}
\subsection{Overview of the estimates in different regions of $\Z$. Main lemmas}
\label{subsection_lemmas}
In this subsection, we will state several main lemmas that will be used in the proof of the main Theorem \ref{th_main}. Each lemma will establish regularity of a certain class of points of $\Z$, with different exponents and length of the intervals. In the next subsection, we will apply Lemmas \ref{lemma_dominating} and \ref{lemma_terminating}, as well as Remark \ref{rem_same_exponent}, to show how localization follows from these lemmas. In the next section, we will prove these lemmas using the large deviation Theorem \ref{th_ldt}.

Let
$$
\frac{p_k}{q_k}\to \alpha,\quad k\to +\infty,
$$ 
be the sequence of continued fraction approximants to $\alpha$. Recall
$$
L(E)>\beta=\limsup\limits_{k\to +\infty}\frac{\log q_{k+1}}{q_k}.
$$
As a consequence, we can state
$$
q_{k+1}\le e^{(\beta+o(1))q_k}\le e^{(L(E)-o(1))q_k},\quad k\to +\infty,
$$
where $o(1)$ only depends on $\alpha$. 

Fix a large constant $C_d$ (say, $C_d=100$). As usual for the arithmetic localization method, the location, size, and exponents of the good intervals will largely dependent on whether $q_{k+1}\le q_k^{C_d}$ (Diophantine transition) or $q_{k+1}>q_k^{C_d}$ (Liouville transition). The fact that $\beta(\alpha)>0$ implies that there are infinitely many scales on which the latter case happens.

In the main lemmas, $o(1)$ is considered as $n\to +\infty$ (or, equivalently, $q_k\to +\infty$) for fixed $\sigma$, $\tau$, $C_{d}$, $\alpha$, $f$, and $E$, but the estimates can be made uniform in $E$ on any finite energy interval on which $L(E)-\beta(\alpha)$ is bounded from below by a positive constant.

The following proposition, is essentially, obtained in \cite{Kach} and partially in \cite{JK}, and is sufficient to establish localization in the Diophantine case. Similar arguments can also be extended for some situations with $\beta>0$ \cite{Zhu}. Even though this regime will be covered later by Lemmas \ref{lemma_diophantine_diophantine} and \ref{lemma_liouville_diophantine}, we include the proof mainly to introduce the reader, who may be already familiar with \cite{JK,Kach}, to the current notation.
\begin{prop}
\label{prop_main_basecase}
Assume that $q_{k+1}\le q_k^{C_d}$. Then, every point $m\in \Z$ with 
$$
|m|\in J=[\lfloor q_k/2\rfloor+1; q_{k+1}-\lfloor q_k/2\rfloor-1].
$$
is $(L(E)-o(1),q_k)$-regular.
\end{prop}

For the main cases, we will need to introduce additional notation. In all lemmas below, we will use a fixed small parameter 
$$
\tau=\tau(E,f,\beta)>0.
$$
Find the largest $q_{k-k_0}$ and then the largest possible $s$ such that
\beq
\label{eq_definition_s}
2s q_{k-k_0}\le \tau q_k.
\eeq
\begin{lemma}[Diophantine to Diophantine transition]
\label{lemma_diophantine_diophantine}
Assume that $q_{k+1}\le q_k^{C_d}$ and $s\le q_{k-k_0}^{C_d}$. Let $s':=\lfloor s/10\rfloor$. Then, every $m\in \Z$ with $100 \tau q_k<|m|<3 q_k/2$ is $(L(E)-o(1),2 s' q_{k-k_0}-1)$-regular.
\end{lemma}

\begin{lemma}[Liouville to Diophantine transition]
\label{lemma_liouville_diophantine}
Assume that $q_{k+1}\le q_k^{C_d}$, $s>q_{k-k_0}^{C_d}$.
Let $s':=\lfloor s/10\rfloor$. Then, every $m$ with $100 \tau q_k<|m|<3 q_k/2$ is $(L(E)-o(1),2 s' q_{k-k_0}-1)$-regular.
\end{lemma}
\begin{remark}
\label{rem_length_intervals_diophantine} In both Diophantine cases,
due to the definition of $s'$, the length of the good interval will be between $\tau q_k/20$ and $\tau q_k$.
\end{remark}
We will now address the case of Liouville transition $q_{k+1}>q_k^{C_d}$. For $\tau$ as above, let
\beq
\label{eq_bk_def}
b_k:=\lfloor \tau q_k\rfloor.
\eeq
We will consider integer points $m\in [-q_{k+1};q_{k+1}]$, and call them {\it resonant} if $\dist(m,q_k \Z)\le b_k$. Otherwise, they will be called {\it non-resonant}. For $\ell\in \Z$, let
$$
R_\ell:=[\ell q_k-b_k;\ell q_k+b_k].
$$
We will have three sub-cases. The first two cases, covering non-resonant points, are largely based on the same ideas as in Lemmas \ref{lemma_diophantine_diophantine} and \ref{lemma_liouville_diophantine} (in a way, the structure of the operator outside of a nearby resonant region somewhat resembles the structure in the Diophantine case near the origin). The third case, with substantially different analysis, covers the resonant regions.
\begin{lemma}[Diophantine to non-resonant transition]
\label{lemma_nonresonant_diophantine}
Suppose that $q_{k+1}>q_k^{C_d}$. Define $s$ as in $\eqref{eq_definition_s}$ and suppose that $s\le q_{k-k_0}^{C_d}$. Then, every non-resonant point $m\in [-q_{k+1};q_{k+1}]\setminus \cup_{\ell}R_{\ell}$ is $(L(E)-o(1),2s q_{k-k_0}-1)$-regular.
\end{lemma}
\begin{lemma}[Liouville to non-resonant transition]
\label{lemma_nonresonant_liouville}
Suppose that $q_{k+1}>q_k^{C_d}$. Define $s$ as in $\eqref{eq_definition_s}$ and suppose that $s\ge q_{k-k_0}^{C_d}$. Let $s':=\lfloor s/10\rfloor$. Then, every non-resonant point $m\in [-q_{k+1};q_{k+1}]\setminus \cup_{\ell}R_{\ell}$ is $(L(E)-o(1),2s' q_{k-k_0}-1)$-regular.
\end{lemma}
The final main lemma covers the resonant points.
\begin{lemma}[Resonant case]
\label{lemma_resonant}
Suppose that $q_{k+1}>q_k^{C_d}$. Then, every resonant point $m\in R_{\ell}$ with $0<|\ell|\le \frac{q_{k+1}}{10 q_k}$ is $\l(\psi,\frac{1}{q_k}\log |\ell|+L(E)-\beta_k-3 \tau,r(1+O(\tau))\r)$-regular, where $r=2q_k$ or $r=4 q_k$.
\end{lemma}
\begin{remark}
\label{rem_length_intervals_liouville}
In the Liouville case, the lengths of good intervals are between $\tau q_k/20$ and $4 q_k(1+O(\tau))$.
\end{remark}
\subsection{Proof of Theorem \ref{th_main}}
\label{subsection_proof_main}
We will now assume that, for sufficiently large $k$, Lemmas
\ref{lemma_diophantine_diophantine},\ref{lemma_liouville_diophantine},\ref{lemma_nonresonant_diophantine},\ref{lemma_nonresonant_liouville}
and \ref{lemma_resonant}  all hold, so the corresponding good intervals exist. Let $\psi$ be a generalized eigenfunction satisfying \eqref{eq_initial_condition}: as a reminder,
$$
H(x)\psi=E\psi,\quad |\psi(0)|^2+|\psi(1)|^2=1,\quad |\psi(m)|\le C_{\psi}(1+|m|).
$$
Our first goal will be to establish exponential decay of $\psi$. Let $m\in \mathbb Z$. We will estimate $|\psi(m)|$ in the cases of Diophantine and Liouville transitions. There will be some overlap between the cases. With $\tau$ as above, one can assume, without loss of generality, that $m>0$.

Suppose that $\sqrt{\tau}q_k\le m\le \sqrt{\tau}q_{k+1}$. In the Diophantine case, all good intervals have exponents $L(E)-o(1)$, and the lengths between $\tau q_k/20$ and $q_k$. Apply Lemmas \ref{lemma_dominating} and \ref{lemma_terminating} in the setting of Remark \ref{rem_same_exponent} with
$$
[N_1;N_2]:=[\tau q_k;q_{k+1}/10], \quad [N_1^-;N_2^+]:=[0;q_{k+1}/10 + q_k],\quad L:=\tau q_k/20,\quad \mu:=L(E)-o(1),
$$
with good intervals obtained from Lemmas \ref{lemma_diophantine_diophantine} and \ref{lemma_liouville_diophantine}. We have, from \eqref{eq_poisson_applied2},
$$
\psi(m)\le C_{\psi}(1+|q_k|)e^{-(L(E)-o(1))|m-N_1|}+C_{\psi}(1+|q_{k+1}|)e^{-(L(E)-o(1))|m-N_2|},\quad m\in [\tau q_k,q_{k+1}/10].
$$
In view of the original assumption $\sqrt{\tau} q_k\le m\le \sqrt{\tau} q_{k+1}$ and the choice of $\tau$ with, say, $\tau</400$, we have
$$
m-N_1\ge m-\tau q_k\ge (1-\sqrt\tau)m,\quad N_2-m\ge \frac{1}{10 \sqrt{\tau}}m-m\ge m.
$$
As a consequence, we have
\beq
\label{eq_decay_dio}
|\psi(m)|\le e^{-((L(E)-o(1))(1-\sqrt{\tau})m},\quad \sqrt{\tau} q_k\le m\le \sqrt{\tau} q_{k+1}.
\eeq
Here we absorbed the Schnol pre-factors into $o(1)$, making it dependent on $\psi$. Since $\tau$ can be made arbitrarily small, it can also be absorbed into $o(1)$ as the last step.

In the Liouville case, we will apply Lemmas \ref{lemma_dominating} and \ref{lemma_terminating} with good intervals obtained from Lemmas \ref{lemma_nonresonant_diophantine}, \ref{lemma_nonresonant_liouville}, and \ref{lemma_resonant}. Recall that the lengths of these intervals are between $\tau q_k/20$ and $4 q_k(1+O(\tau))$ which, again, makes $L=\tau q_k/20$ and $\mu=L(E)-\beta_k-o(1)-O(\tau)$. The rest is similar to the previous case: for
$$
m\in [\tau q_k;q_{k+1}/10]:=[N_1;N_2]\subset [N_1^-;N_2^+]=:[0;q_{k+1}/10 + q_k],
$$
one has
\beq
\label{eq_decay_lio_simple}
|\psi(m)|\le e^{-((L(E)-\beta_k-o(1))(1-\sqrt{\tau})m},\quad \sqrt{\tau} q_k\le m\le \sqrt{\tau} q_{k+1}.
\eeq
By combining \eqref{eq_decay_dio} and \eqref{eq_decay_lio_simple}, we show that every generalized eigenfunction decays exponentially, thus completing the proof of Theorem \ref{th_main}.\,\qed

\subsection{Improvements in the exponential decay of eigenfunctions} 
\label{subsection_improvements}
We will now discuss some improvements to the decay of eigenfunctions in the case of the Liouville transition. As above, let
$$
[\tau q_k;q_{k+1}/10]:=[N_1;N_2]\subset [N_1^-;N_2^+]=:[0;q_{k+1}/10 + q_k],
$$
and assume that $m\in [\sqrt{\tau}q_k;\sqrt{\tau} q_{k+1}]$ (again, without loss of generality, assume $m>0$). In view of Remark \ref{rem_same_exponent}, we would like to find lower bounds on the weights of the paths that start at $m$ and terminate at $[N_1^-;N_1)$ and $(N_2;N_2^+]$. Since the distance from $q_{k+1}/10$ to $q_{k+1}$ is much larger than the distance to the origin, it is easy to see even from simpler estimate \eqref{eq_poisson_applied} that, for $\tau \ll 1$ the contribution of the paths terminating at $(N_2;N_2^+]$ contains a much smaller exponential factor, and can be absorbed into the contribution of the paths terminating near the origin (the first term in \eqref{eq_poisson_applied}).

If $m$ is in a resonant zone $R_\ell$, then the worst possible case (meaning the case with lowest theoretically possible weight) would be if all good intervals are from Case (1) of Step 2.2, and the next iteration point always ends up in the resonant zone, since in this case $\mu_m$ will be the smallest. Such a path would travel from $R_\ell$ to $R_{\ell-1},R_{\ell-2},\ldots$, and then finish near origin where it would terminate. It is easy to see that any deviation from this scenario (either by ending up in a non-resonant zone, or by having Case (2) of Step 2.2) would lead to using good intervals with larger exponents which would result in a larger weight (say, modulo an $O(\tau)w(P)$ correction resulting in accounting for endpoints). If $m$ is not in a resonant zone, then the most efficient way to get the lowest weight would be to use non-resonant good intervals to arrive to one of the two nearest resonant zones, and then proceed as before. Overall, this leads to two paths $Q_1$, $Q_2$ whose weights could serve as lower bounds:
$$
w(Q_1)=\log (\ell!)+(L(E)-\beta_k-O(\tau)-o(1))\ell q_k+(L(E)-o(1))|m-\ell q_k|.
$$
$$
w(Q_2)=\log ((\ell+1)!)+(L(E)-\beta_k-O(\tau)-o(1))(\ell+1)q_k+(L(E)-o(1))|(\ell+1)q_k-m|.
$$
Note that, modulo the difference that can be absorbed into $o(1)$, we can replace both $\log (\ell!)$ and $\log (\ell+1)!)$ by $\ell \log (\ell+1)$.

In order to rewrite these expressions in a unified form, let $s:=\ell$ if the minimal weight is attained on $Q_1$, and $s:=\ell+1$ if attained on $Q_2$. Let $Q$ be the corresponding path out of $Q_1$, $Q_2$ with the smallest weight. Then it is easy to see that
\begin{multline*}
w(Q)=L(E)\dist(m,R_{s})+s\log (s+1) +(L(E)-\beta_k)s q_k+L(E)m (o(1)+O(\tau))\\
=\l(L(E)-\beta_k+\frac{\log m}{q_k}\r)s q_k+L(E)\l(\dist(m,R_{s})+m (o(1)+O(\tau))\r).	
\end{multline*}
Similar estimates have been obtained in \cite{JL2}, \cite{JHY} for the almost Mathieu operators and the Maryland model. From Lemmas \ref{lemma_dominating}, \ref{lemma_terminating}, Schnol's lemma, Poisson's formula, and after absorbing various corrections into $O(\sqrt{\tau})$, we finally obtain
\beq
\label{eq_improvement_1}
|\psi(m)|\le e^{-w(Q)(1-O(\sqrt{\tau}))},\quad \sqrt{\tau} q_k\le m\le \sqrt{\tau} q_{k+1}.
\eeq
Note that, in all cases
\beq
\label{eq_improvement_2}
w(Q)\ge (L(E)-\beta_n-O(\tau)-o(1))m,
\eeq
which was already noted when obtaining \eqref{eq_decay_lio_simple}. However, $\eqref{eq_improvement_1}$ is significantly more precise than \eqref{eq_improvement_2}. Indeed
let us now discuss some situations where the decay rate can get better than in \eqref{eq_improvement_2}. From now on, we will discuss all decay rates modulo a factor $(1+O(\sqrt{\tau}))$.

For $s=0$, we have $w(Q)=m L(E)$. This corresponds to the case when $m$ is between $R_0$ and $R_1$ and not too close to $R_1$. As a consequence, we reproduce the Lyapunov decay in the non-resonant case. Similar improvements will happen every time between resonances; however, the relative contribution of this improvements will get smaller and smaller as $s$ gets larger.

The worst possible case is when $m$ either close to $R_{s}$ or $s$ is large enough so that $L(E)\dist(m,R_s)$ does not give a meaningful improvement, but $s$ is too small for the term $s \log (s+1)$ to give a noticable improvement. In this case, the bound on the decay rate of $\psi$ is $L(E)-\beta_k$, as in \eqref{eq_decay_lio_simple}.

Once $\ell$ becomes comparable to $\frac{q_{k+1}}{q_k}$ on the logarithmic scale (that includes a neighborhood of $\sqrt{\tau}q_{k+1}$ in the log scale), the ratio $\frac{\log m}{\log q_k}$ cancels out $\beta_k$, and we return to the decay rate $L(E)$. One can state that, on the boundary of the region $[\sqrt\tau q_k;\sqrt{\tau}q_{k+1}]$ the decay rate is $L(E)$ as in the non-resonant case.
\subsection{Some known improvements: uniform localization and localization for almost every phase}
For the convenience of the reader, we will discuss some other improvements mentioned in \cite{JK,Kach}. The first one is (locally) uniform localization in the following sense.
\begin{cor}
\label{cor_uniform_localization}
Under the assumptions of Theorem $\ref{th_main}$, for every $\delta>0$ and $M>0$ there exists a sequence
$$
\ep_n=\ep_n(M,f,\alpha,\delta)\to +0\quad \text{as}\quad n\to +\infty
$$
such that, for every solution of the eigenvalue equation
$$
H(x)\psi=E\psi,\quad \|\psi\|_{\ell^{\infty}(\Z)}=1,\quad L(E)>\beta(\alpha)+\delta,\quad |E|<M,
$$
there exists $n_0=n_0(\psi)$ such that
$$
|\psi(n-n_0)|\le \exp\{-(L(E)-\beta-\ep_n)|n-n_0|\}.
$$
\end{cor}
\begin{proof}
In the proof of Theorem \ref{th_main}, the only step that would not immediately lead to the above uniform bounds was the choice of a generalized eigenfunction $\psi$ with 
\beq
\label{eq_schnol_assumption}
|\psi(n)|\le C(1+|n|).
\eeq
 In this form, $C$ must depend on $\psi$, since the localization center may be far away from the origin. However, once $\psi\in \ell^2$, one can normalize it in $\ell^{\infty}$ and, by changing $x$, reduce it to the case when $\psi$ attains its maximal value at the origin, therefore satisfying (trivially) \eqref{eq_schnol_assumption} with $C=1$.
\end{proof}
\begin{cor}
\label{cor_almost_every_x}
Under the assumptions of Theorem $\ref{th_main}$, let $A\subset \R$, and suppose that $L(E)>\beta(\alpha)$ Lebesgue a. e. on $A$. Then, for Lebesgue a. e. $x\in [0,1)$, the set $A$ can only support purely point spectrum of $H(x)$.
\end{cor}
\begin{proof}
Follows from the same argument as in \cite{JK,Kach}: since the integrated density of states is Lipschitz continuous, the ``bad'' set 
$$
\{E\in A\colon L(E)\le \beta(\alpha)\}
$$
must have zero density of states measure, and therefore zero spectral measure for almost every $x\in[0,1)$.
\end{proof}
\section{Regular points and good intervals from the large deviation theorem and the Cramer's rule}
Throughout this section, we will fix $E$ and denote
$$
P_n(x):=P_n(x,E).
$$
Let $[n_1;n_2]$ be an integer interval of length $n$ (that is, $n_2=n_1+n-1$). We will use the Cramer's rule approach to estimating Green's functions, originally appeared in \cite{Lana94} and developed in \cite{Lana99} and later works:
\beq
\label{eq_cramer_1}
|(H_{[n_1,n_2]}(x)-E)^{-1}(n_1,m)|=\l|\frac{P_{n_2-m}(x+(m+1)\alpha)}{P_n(x+n_1\alpha)}\r|,
\eeq
\beq
\label{eq_cramer_2}
|(H_{[n_1,n_2]}(x)-E)^{-1}(m,n_2)|=\l|\frac{P_{m-n_1}(x+n_1\alpha)}{P_n(x+n_1\alpha)}\r|
\eeq
Recall that $m\in \Z$ is $(\mu,n)$-regular if there exists an integer interval $[n_1;n_2]$ (a good interval around $m$) with 
$$
m\in [n_1;n_2],\quad n_2=n_1+n-1,\quad |m-n_j|\ge \sigma n=\sigma|n_2-n_1+1|, ,\quad j=1,2,
$$
and
$$
|(H_{[n_1,n_2]}-E)^{-1}(n_j,m)|\le e^{-\mu|m-n_j|},\quad j=1,2.
$$
The main idea of \cite{Lana94,Lana99} is that, for a fixed $m\in \Z$, one can search for good intervals among translations of a fixed interval of length $n$ with above properties. Since the translations correspond to different values of $x+n_1\alpha$, one can use large deviation theorems like Theorem \ref{th_ldt} in order to avoid small denominators. This step is usually the most model-specific.

We will start from the following lemma, see \cite{Olympiad}. We include the proof from \cite{Olympiad} for the convenience of the reader, since an English translation may not be easily available.
\begin{lemma}
\label{lemma_polynomial}
Let $p=x^n+a_{n-1}x^{n-1}+\ldots+a_0$ be a polynomial with real coefficients, and $x_0,\ldots,x_{n}\in \R$ with $|x_i-x_j|\ge d$ for $i\neq j$. Then, at least one of the above points satisfies
\beq
\label{eq_polynomial_factorial}
|p(x_j)|\ge (d/2)^n n!
\eeq
\end{lemma}
\begin{proof}
Let
$$
D_j:=\prod\limits_{i=0,i\neq j}^{n}(x_j-x_i).
$$
Without loss of generality, possibly after reordering the points and rescaling, one can assume that $d=1$ and $x_0<x_1<x_2\ldots<x_n$, which implies
$$
|x_i-x_j|\ge |i-j|,\quad |D_j|\ge j! (n-j)!.
$$
By comparing the leading coefficients in the usual Lagrange interpolation formula, we have 
$$
1=\l|\sum_{j=0}^{n}\frac{p(x_j)}{D_j}\r|\le \frac{\max_j|p(x_j)|}{n!}\sum_{j=0}^{n}\binom{n}{j}=\frac{2^n}{n!}\max_j|p(x_j)|.\,\qedhere
$$
\end{proof}
Recall that $P_n(x)$ is represented as a product
$$
P_n(x,E)=P_n^{<B_-}(x)P_n^{\mathrm{mid}}(x)P_n^{>B_+}(x).
$$
We will use Lemma \ref{lemma_polynomial} in order to avoid small values $P_n^-(x_j)$ (and, therefore, small values of $P_n(x_j)$) among certain collections of sampling points $x_j$.
\begin{lemma}
\label{lemma_sampling_points}
Let $I\subset \R$ be an interval with $|I|<(10\gamma)^{-1}$, and $x_0,\ldots,x_l\in I$ satisfy $|x_i-x_j|\ge d>0$ for $i\neq j$. Assume that $I$ contains at most $t\le l$ intersection points $\eqref{eq_intersection_points}$ for $P_n$, and all other intersection points $z$ of $P_n$ satisfy 
\beq
\label{eq_no_intersection_points_outside}
\dist(z,I)\ge \gamma^{-1}B_-.
\eeq 
Then, for at least one $j\in [1;l]$, we have
\beq
\label{eq_large_value}
|P_n^{<B_-}(x)|\ge \gamma^t (d/2)^t t!.
\eeq
\end{lemma}
\begin{proof}
Let $z_1,\ldots,z_t$ be the intersection points in $I$. From \eqref{eq_p_minus_twosided}, we have
\beq
\label{eq_p_lower}
|P_n^{<B_-}(x)|\ge \gamma^t \prod_{j=1}^t\|x-z_j\|=\gamma^t \prod_{j=1}^t|x-z_j|,\quad x\in I.
\eeq
Note that \eqref{eq_no_intersection_points_outside} implies that there is no contribution from intersection points outside of $I$. While $P_n^-(x)$ may have less than $t$ factors, the assumption $|I|<(10\gamma)^{-1}$ guarantees that each of the factors in the right hand of \eqref{eq_p_lower} does not exceed $1$, and therefore the inequalities still hold. Afterwards, \eqref{eq_large_value} follows from Lemma \ref{lemma_polynomial} applied to $p(x)=(x-z_1)\ldots(x-z_t)$.
\end{proof}
We are now ready to fix some parameters that are used in Theorem \ref{th_ldt} in order to start applying the above results to finding good intervals. As a reminder, $q(n)$ is the largest denominator among continued fraction approximants of $\alpha$ with
$$
n=s q(n)+r(n),\quad 0\le |r(n)|\le \sqrt{n},\quad \frac{s+r}{n}\le \frac{1}{q(n)}+\frac{|r(n)|}{n}\le \frac{1}{q(n)}+\frac{1}{\sqrt{n}},
$$
where the last inequalities follow from the first two properties. Recall that $C_d=100$ is a fixed large constant, and set
\beq
\label{eq_choices_B}
B_-(n):=q(n)^{-2 C_d+1},\quad B_+(n):=q(n),\quad B_+'(n):=q(n)+4.
\eeq
We are now ready to state the main theorem of this subsection, which relates avoiding small denominators by Lemma \ref{lemma_sampling_points} with finding regular intervals.
\begin{lemma}
\label{lemma_large_numerators}
Let $0\le n'\le n$, and suppose that 
\beq
\label{eq_intervals_included}
[h';h'+n']\subset [h;h+n]
\eeq
are two integer intervals. Then
$$
\l|\frac{P_{n'}^{>B_+'(n)}(x+h'\alpha)}{P_{n}^{>B_+(n)}(x+h\alpha)}\r|\le \l(\frac{q(n)+4}{q(n)}\r)^{n\l(\frac{1}{\sqrt{n}}+\frac{C(f,E)}{q(n)}\r)}\le e^{n\log(2)\l(\frac{1}{\sqrt{n}}+\frac{C(f,E)}{q(n)}\r)},
$$
where $C(f,E)$ is uniformly bounded on every compact energy interval.
\end{lemma}
\begin{proof}
\eqref{eq_p_plus_factors}.
Due to \eqref{eq_intervals_included}, each diagonal entry of the operator $H_{[h',h'+n']}(x)$ is also a diagonal entry of $H_{[h,h+n]}(x)$. The estimate follows from the two-sided bound \eqref{eq_p_plus_twosided_twosided}, applied as a lower bound on the denominator and an upper bound on the numerator. Here, $\frac{q(n)+4}{q(n)}$ estimates the ratio between an upper and lower bound on the same eigenvalue, and the exponent comes from the estimate on the number of factors \eqref{eq_p_plus_factors}. Note that the denominator may have more factors, however, each additional factor is larger than $1$ and therefore can only make the estimate stronger.
\end{proof}
The following is the main application of the large deviation theorem and its corollaries, relating avoiding small denominators using Lemma \ref{lemma_sampling_points} with finding a good interval. Note that the expression for $R(n)$ below is not of any particular importance. However, it is essential that $R(n)\to 0$ as $n\to +\infty$, uniformly on any compact energy interval.
\begin{theorem}
\label{th_main_cramer}
Under the assumptions of Lemma $\ref{lemma_sampling_points}$, suppose, in addition, that $m\in \Z$, $n\in \N$, and
$$
x_j=x+h_j\alpha,\quad h_j\in \Z, \quad h_j+\sigma n\le m\le h_j+n-\sigma n,\quad j=0,\ldots,l.
$$
Then, for at least one $j\in [0;l]$, one has
\beq
\label{eq_cramer1_applied}
\l|\l(H_{[h_j,h_j+n]}(x)-E\r)^{-1}(m,h_j)\r|\le \frac{e^{-|m-h_j|L(E)}}{\gamma^t (d/2)^t t!}e^{nR(n)}=\frac{e^{-|m-h_j|L(E)(1-o(1))}}{\gamma^t (d/2)^t t!},
\eeq
\beq
\label{eq_cramer2_applied}
\l|\l(H_{[h_j,h_j+n]}(x)-E\r)^{-1}(m,h_j+n)\r|\le \frac{e^{-|m-(h_j+n)|L(E)}}{\gamma^t (d/2)^t t!}e^{nR(n)}=\frac{e^{-|m-(h_j+n)|L(E)(1-o(1))}}{\gamma^t (d/2)^t t!},
\eeq
where
$$
R(n)=\frac{\log 2}{\sqrt{n}}+\frac{C(f,E)}{q(n)}+2 L_{\mathrm{corr}}(E,q(n)^{-199},q(n))+400 C_d\frac{\log q(n)}{q(n)}+400 C_d\frac{|r(n)|\log q(n)}{n},
$$
and $C(f,E)$ is uniformly bounded on any compact energy interval.

If, in addition,
$$
\gamma^t (d/2)^t t!\ge e^{-n o(1)},
$$
then the interval $[h_j,h_j+n]$ is $(L(E)-o(1))$-regular.
\end{theorem}
\begin{proof}
We will only estimate \eqref{eq_cramer1_applied}, since \eqref{eq_cramer2_applied} is similar. From the Cramer's rule \eqref{eq_cramer_1}, we have
\begin{multline}
\label{eq_many_cramer}
\l|\l(H_{[h_j,h_j+n]}(x)-E\r)^{-1}(m,h_j)\r|=\frac{|P_{m-h_j}(x+h_j\alpha)|}{|P_n(x+h_j\alpha)|}\\
=\l|\frac{P_{m-h_j}^{<B_-(n)}(x+h_j\alpha)P_{m-h_j}^{\mathrm{mid}}(x+h_j\alpha)}{P_n^{\mathrm{mid}}(x+h_j\alpha)}\r|
\cdot\l|\frac{1}{P_n^{<B_-(n)}(x+h_j\alpha)}\r|\cdot \l|\frac{P_{m-n_1}^{>B_+'(n)}(x+h_j\alpha)}{P_{n}^{>B_+(n)}(x+h_j\alpha)}\r|.
\end{multline}
In the first factor, the numerator is estimated from above using Corollary \ref{cor_uniform_upper}, and the denominator from below using Corollary \ref{cor_ldt_specific_s} from the large deviation theorem \ref{th_ldt}. The middle factor is estimated from below using Lemma \ref{lemma_sampling_points}. Finally, the last factor is estimated from above using Lemma \ref{lemma_large_numerators}. Here, in the definition of $P_{m-h_j}^{\mathrm{mid}}(x+h_j\alpha)$, it is assumed that the as the upper cut-off point is $B_+'(n)$. 

All correcting factors are gathered in $R(n)$, which is written in a somewhat weakened form in order to provide an estimate both on \eqref{eq_cramer1_applied} and \eqref{eq_cramer2_applied}. Note that, in the last equalities of both \eqref{eq_cramer1_applied}, \eqref{eq_cramer2_applied}, we used the fact that $\dist(m,\partial[h_j;h_j+n])\ge \sigma n$. The ending conclusion follows directly from the definition of regularity.
\end{proof}

\section{Proofs of the main lemmas}\label{proofs}
We start from some overview of the relations between proofs of the main lemmas.
\begin{enumerate}
\item Recall that the definition of a regular point includes a parameter $\sigma>0$, see \eqref{eq_length_regular_interval}. In all proofs, we will assume that some small value (say, $\sigma=1/5$) is fixed. One also has a constant $C_d=100$ (say) determining the difference between the Diophantine and Liouville transitions.
\item As mentioned before, the proof of Proposition \ref{prop_main_basecase} is, essentially, contained in \cite{JK,Kach}, since it only requires LDT on scales of the form $n=q_k$. We include the complete proof, since the next lemma \ref{lemma_diophantine_diophantine} will be based on a somewhat related argument. 
\item Lemmas \ref{lemma_liouville_diophantine} and \ref{lemma_nonresonant_liouville} are the most difficult, since the estimates in Theorem \ref{th_ldt} are the weakest (one has $s\gg 1$, and therefore $q\ll n$). The main novel idea is to get on a somewhat smaller scale $2s'q_{k-k_0}$, where the sampling points of the trajectory naturally split into clusters. Then, one can use Lemma \ref{lemma_sampling_points} locally in a neighborhood of one cluster.
\item With some care, one could combine Lemmas \ref{lemma_liouville_diophantine} and \ref{lemma_nonresonant_liouville} into one unified regime, however, we believe that the reader may benefit from observing the main idea in a somewhat more simple setting of Lemma \ref{lemma_liouville_diophantine} first. Additionally, such modification would create an additional regime which will need to be covered by an extended version of Proposition \ref{prop_main_basecase}.
\end{enumerate}

\subsection{The base Diophantine case: proof of Proposition \ref{prop_main_basecase}} In the Diophantine case, we assume $q_{k+1}\le q_k^{C_d}$. Assume first that $q_k$ is even, and consider the following set of integers:
\beq
\label{eq_definition_J_base}
J:=\left([-q_k/2;-1]-\lfloor 2\sigma q_k\rfloor\right)\cup \left([-q_k/2;0]-\lfloor 2\sigma q_k\rfloor+m\right).
\eeq
For odd $q_k$, replace $-q_k/2$ by $-\lfloor q_k/2\rfloor$ and add one extra point to any of the subsets (by increasing the length of any of the intervals by $1$). As a consequence, if $m>q_k/2+1$, the set $J$ contains $q_k+1$ distinct integer points. Assuming additionally that $m\le q_{k+1}-q_k/2$, Proposition \ref{prop_rotation_trajectory} together with the original assumption implies
$$
\|(j_1-j_2)\alpha\|\ge \frac{1}{2q_{k+1}}\ge \frac12 q_k^{-C_d},\quad j_1,j_2\in J,\,\,j_1\neq j_2;\quad q_k/2+1<m<q_{k+1}-q_k/2.
$$
In other words, for every $x\in \R$, the points $\{x+j\alpha\colon j\in J\}$ are $\frac12 q_k^{-C_d}$-separated modulo $1$. Consider the determinant $P_{q_k}(x)$. Since it has at most $q_k$ intersection points, there is at least one value $j\in J$ with $x+j\alpha$ such that the interval
$$
I:=\l[x+j\alpha-\frac14 q_k^{-C_d};x+j\alpha+\frac14 q_k^{-C_d}\r]
$$
does not contain any intersection points of $P_{q_k}$ modulo $1$. Apply Lemma \ref{lemma_sampling_points} in the trivial case, where $\ell=1$ and $x_0$ is the only sampling point, and $t=0$ (no intersection points), which guarantees
$$
P_{q_k}^{<B_-(q_k)}(x+j\alpha)\ge 1
$$
(in fact, there will be simply no factors in $P_{q_k}^{<B_-(q_k)}(x+j\alpha)$). Theorem \ref{th_main_cramer} implies that the interval 
$$
Q=[j;j+q_k-1]
$$ 
is $(L(E)-o(1))$ good: that is, 
\begin{eqnarray}
\label{eq_green_good_temp1}
|(H_{Q}(x)-E)^{-1}(j,m')|\le e^{-(L(E)-o(1))|m'-j|};\\
|(H_{Q}(x)-E)^{-1}(m',j+q_k-1)|\le e^{-(L(E)-o(1))|j+q_k-1-m'|}
\end{eqnarray}
for all $m'$ satisfying
$$
m'\in Q,\quad \dist(m',\partial Q)\ge \sigma |Q|=\sigma q_k.
$$

If $j$ is was the first interval of $J$ from \eqref{eq_definition_J_base}, it would imply that $Q$ is good with respect to the origin (so that the above estimates hold with $m'=0$, making the origin $(L(E)-o(1),q_k)$-regular). However, the latter is not allowed by Proposition \ref{prop_origin}. As a consequence, $j$ must belong to the second interval of $J$. Under the above assumption on $\sigma$, the interval $Q=[j;j+q_k-1]$ will be $(L(E)-o(1))$-good with respect to $m$, and therefore $m$ will be $(L(E)-o(1),q_k)$-regular as stated.\,\,\qed

Note that, if the next transition from $q_{k+1}$ to $q_{k+2}$ is also Diophantine, one can push $m$ all the way to $q_{k+2}-q_k/2$, since the distance between $j_1\alpha$ and $j_2\alpha$ modulo $1$ will still have a polynomial lower bound in $q_k$. Therefore, one does have ample overlap between scales and the above argument (perhaps under a modification of the choice of $B_1(n)$ in \eqref{eq_choices_B}) is indeed sufficient to obtain localization for Diophantine $\alpha$, as done in \cite{JK,Kach}.

\subsection{Diophantine to Diophantine transition: proof of Lemma \ref{lemma_diophantine_diophantine}}
As a reminder, we have fixed a small number $\tau>0$ and choose $\tau_k$ such that $b_k:=\tau_k q_{k}$ is the closest integer to $\tau q_{k}$. Find the largest $q_{k-k_0}$ with $2 q_{k-k_0}\le b_k$, and  then the largest $s>0$ with $2sq_{k-k_0}\le b_k$. In Lemma \ref{lemma_diophantine_diophantine} we assume $s\le q_{k-k_0}^{C_d}$. In this case, let
$$
J:=\left([-s q_{k-k_0};-1]-\lfloor s q_{k-k_0}/2\rfloor\right)\cup \left([-s q_{k-k_0};-1]-\lfloor s q_{k-k_0}/2\rfloor+m\right).
$$
For $m\ge s q_{k-k_0}$, the set $J$ contains $2 s q_{k-k_0}$ distinct integer points. The rest of the argument repeats that from base Diophantine case from Proposition \ref{prop_main_basecase}: that is, the corresponding points on the circle will have the following separation property:
$$
\|(j_1-j_2)\alpha\|\ge \frac{1}{2q_{k+1}}\ge \frac12 q_k^{-C_d},\quad j_1,j_2\in J,\,\,j_1\neq j_2;\quad s q_{k-k_0}\le m\le q_k-s q_{k-k_0}-1,
$$
and we have for some $j\in J$, as above,
$$
P^{<B_-(n)}_{n}(x+j\alpha)\ge 1,\quad n=2sq_{k-k_0}-1,
$$
which implies that the conclusion of Theorem \ref{th_main_cramer} also holds, making an interval
$$
Q=[j,j+n-1]=[j,j+2s q_{k-k_0}-2]
$$
good. In view of Proposition \ref{prop_origin}, the corresponding $j$ has to be from the second interval of $J$. Since $sq_{k-k_0}<\tau q_k$, we have that $m$ with $\tau q_k\le m \le (1-\tau) q_k$ is $(L(E)-o(1),2 s q_{k-k_0}-1)$-regular.\footnote{Note that the above considerations also include the range of $m$ already covered in Proposition \ref{prop_main_basecase}. One could also obtain it by the means of that proposition alone, by considering the transition from $q_{k-k_0}$ to $q_{k+1}$ directly and using the estimate $q_{k+1}\le q_{k-k_0}^{C_d^2}$. In this case we would have good intervals of length $q_{k-k_0}$ which would also be sufficient for the purposes of localization. This would also involve a minor change in \eqref{eq_choices_B} replacing $C_d$, say, by $C_d^3$}\,\,\qed

\subsection{Liouville to Diophantine transition: proof of Lemma \ref{lemma_liouville_diophantine}} Define $\tau_k$, $b_k$, $k_0$, $q_{k-k_0}$, and $s$ as above, and consider now the case $s>q_{k-k_0}^{C_d}$. Let $s'=\lfloor s/10\rfloor$. Define
$$
J:=\left([-s' q_{k-k_0};-1]-\lfloor s' q_{k-k_0}/2\rfloor\right)\cup \left([-s' q_{k-k_0};-1]-\lfloor s' q_{k-k_0}/2\rfloor+m\right).
$$
For $A\subset \Z$, we will use the notation
$$
A\alpha:=\{m\alpha\colon m\in A\}.
$$
Sometimes, $A\cdot\alpha$ will be used instead for typographical purposes. Let us discuss in more detail the distribution of the $2s'q_{k-k_0}$ points $J\alpha$. Let
$$
J_0:=[0;q_{k-k_0}-1],
$$
$$
J_1:=[0;s' q_{k-k_0}-1]=\cup_{r=0}^{s'-1}(J_0+r q_{k-k_0}).
$$
Note that $J$ is the union of two translated copies of $J_1$:
\beq
\label{eq_definition_j_lio_dio}
J=(J_1-s'q_{k-k_0}-\lfloor s' q_{k-k_0}/2\rfloor)\cup(J_1-s'q_{k-k_0}-\lfloor s' q_{k-k_0}/2\rfloor+m).
\eeq
Proposition \ref{prop_rotation_trajectory} implies that $J_0\alpha$ is a small perturbation of the set 
$$
J_0\cdot\frac{p_{k-k_0}}{q_{k-k_0}}=J_0\cdot \frac{1}{q_{k-k_0}}
$$
modulo 1 (note that the factor $p_{k-k_0}$ only acts as a permutation modulo 1). More precisely, for every point  $y\in J_0\alpha$ there exists a unique $j\in [0;q_{k-k_0}-1]$ with
\beq
\label{eq_lana_C}
\left\|y-\frac{j}{q_{k-k_0}}\r\|<\frac{1}{q_{k-k_0+1}}\le \frac{1}{q_{k-k_0}^{C_d}}
\eeq
By maximality of $q_{k-k_0}$, we have $10 s' q_{k-k_0}\le s q_{k-k_0}\le q_{k-k_0+1}$, and therefore,
\beq
\label{eq_lana_D}
\|r q_{k-k_0}\|\le r\|q_{k-k_0}\|\le \frac{(s'-1)}{q_{k-k_0+1}}\le \frac{1}{10 q_{k-k_0}},\quad 0\le r\le s'-1.
\eeq
As a consequence, the set $J_1\alpha$ is contained in a $\frac{1}{9 q_{k-k_0}}$-neighborhood of the set $J_0\cdot \frac{1}{q_{k-k_0}}$ modulo 1. Using \eqref{eq_lana_C} and \eqref{eq_lana_D}, one can restate it as
\beq
\label{eq_gapped_intervals}
J_1\alpha\subset \bigcup_{j=0}^{q_{k-k_0}-1}\l[\frac{j}{q_{k-k_0}}-\frac{1}{9 q_{k-k_0}};\frac{j}{q_{k-k_0}}+\frac{1}{9 q_{k-k_0}}\r]\,\, \mathrm{mod}\,\,1.
\eeq
The set in the right hand side is a union of $q_{k-k_0}$ intervals of lengths $\frac{2}{9 q_{k-k_0}}$, equally spaced on the circle, with gaps of size $\frac{7}{9 q_{k-k_0}}$ between adjacent intervals. The same is true for each of the sets $(J_1-s'q_{k-k_0}-\lfloor s' q_{k-k_0}/2\rfloor)\alpha$ and $(J_1-s'q_{k-k_0}-\lfloor s' q_{k-k_0}/2\rfloor+m)\alpha$, associated to the left and right halves of $J$, since those are translations of $J_1$. Therefore, by \eqref{eq_definition_j_lio_dio}, $J\alpha$ is contained in a union of two translated copies of the collection of intervals in the right hand side of \eqref{eq_gapped_intervals}, where everything is still considered modulo 1. After combining each interval from the first copy with the closest interval from the second copy, one can see that $J\alpha$ is contained in a union of $q_{k-k_0}$ equally spaced intervals $I_i$, $i=1,\ldots,q_{k-k_0}$ of length (at most) $\frac{13}{18 q_{k-k_0}}$, with distances (at least) $\frac{5}{18 q_{k-k_0}}$ between adjacent intervals. 

Note that the distances $\frac{5}{18q_{k-k_0}}$ and $\frac{13}{18q_{k-k_0}}$ correspond to the extreme case when the intervals associated to $(J_1-s'q_{k-k_0}-\lfloor s' q_{k-k_0}/2\rfloor)\alpha$ do not overlap with intervals associated to $(J_1-s'q_{k-k_0}-\lfloor s' q_{k-k_0}/2\rfloor+m)\alpha$ and together form a family of $2 q_{k-k_0}$ equally spaced intervals of length $\frac{2}{9 q_{k-k_0}}$ (in this case, the nearest intervals should be chosen, say, clockwise). A deviation from this scenario would rotate one family of intervals towards the other, making the size of merged intervals smaller and the gaps between them larger.

Assuming 
$$
s' q_{k-k_0}\le m\le q_k-s' q_{k-k_0}-1,
$$
all points of $J$ will be distinct, each of the intervals $I_i$ will contain exactly $2s'$ points of $J\alpha$, and, by \ref{eq_trajectory_estimates1} and \eqref{eq_trajectory_estimates2}, the distances between any two different points of $J\alpha$ will be bounded from below by $\frac{1}{2 q_k}$. The points $J\alpha$ will be referred to as {\it sampling points}, and the above intervals $I_i$ will be called {\it sampling intervals}.

We now have all ingredients in place to apply Lemma \ref{lemma_sampling_points}. 
 Consider the determinant $P_{2 s' q_{k-k_0}-1}(x)$. It has at most $2 s' q_{k-k_0}-1$ intersection points modulo 1. Note that the $\frac{1}{9 q_{k-k_0}}$-neighborhoods of the sampling intervals are disjoint. As a consequence, there exists a sampling interval $I$ such that its $\frac{1}{9 q_{k-k_0}}$-neighborhood $I'$ contains at most $2s'-1$ intersection points. By construction, $I$ contains $2s'$ sampling points with distances between adjacent points at least $\frac{1}{2 q_k}$. Apply Lemma \ref{lemma_sampling_points} with $x_0,\ldots,x_{2s'-1}$ being those sampling points, the interval $I$ as above, and $d=\frac{1}{2 q_k}$. Note that the choice of $B_-(n)$ with $n=2 s' q_{k-k_0}-1$ in \eqref{eq_choices_B} agrees with the assumptions of the lemma. For some $j\in J$ (with $j\alpha\in I$ modulo 1), and large $k$, this will provide
\beq
\label{eq_verbatim}
|P_{n}^{<B_-(n)}(x+j\alpha)|\ge \gamma^{2s'-1}\l(\frac{1}{4q_k}\r)^{2s'-1} (2s'-1)!,\quad n=2 s' q_{k-k_0}-1.
\eeq
Recall that, by the definition of $s'$, we have $q_k\le \frac{20 s' q_{k-k_0}}{\tau}$. Using the estimate $m!\ge m^m e^{-m}$ and absorbing various terms into $e^{-no(1)}$, we arrive to
\begin{multline*}
|P_{n}^{<B_-(n)}(x+j\alpha)|\ge \frac{(2s')!}{q_k}e^{-n o(1)}\ge \frac{(2s')!}{(2s')^{2s'}}\l(\frac{\tau}{10 q_{k-k_0}}\r)^{2s'}e^{-n o(1)}\ge e^{-n o(1)}.
\end{multline*}
Similarly to the previous cases, Theorem \ref{th_main_cramer} implies that the interval $[j,j+n-1]$ is $(L(E)-o(1))$-regular. Again, if $j$ belonged to the first half of $J$ in \eqref{eq_definition_j_lio_dio}, that would contradicts Proposition \ref{prop_origin} for large $k$. For $\sigma<1/5$, regularity of the interval $[j,j+n-1]$ for any $j$ in the second half of $J$ implies regularity of $m$. Since $10s' q_{k-k_0}<\tau q_k$, we have established $(L(E)-o(1),2s'q_{k-k_0}-1)$ regularity of all $m\in [\tau q_k;(1-\tau)q_k]$.\,\,\qed

\subsection{The Liouville transition: non-resonant regions, some preliminaries}
\label{subsection_nonresonant_preliminaries}
In this subsection, we will introduce some notation and preliminary estimates which will be used Lemmas \ref{lemma_nonresonant_diophantine} and \ref{lemma_nonresonant_liouville}. As a reminder, in both cases we have the Liouville transition $q_{k+1}>q_k^{C_d}$. Earlier, we have fixed a small $\tau>0$ and chose $\tau_k$ such that $b_k:=\tau_k q_{k}$ is the closest integer to $\tau q_{k}$. Integer points $m\in [-q_{k+1}/10;q_{k+1}/10]$ are called {\it resonant} if $\dist(m,q_k \Z)\le b_k$. Otherwise, they are called {\it non-resonant}. 
Similarly to the previous cases, find the largest $q_{k-k_0}$ with $2 q_{k-k_0}\le \dist(m,q_k\Z)$. Then, find the largest $s$ with the similar property $2 sq_{k-k_0}\le \dist(m,q_k\Z)$. From the construction, we have
\beq
\label{eq_scales_relations}
\frac {b_k}{2}=\frac{\tau_k q_k}{2}<2s q_{k-k_0}\le \dist(m,q_k\Z)\le 2 q_{k-k_0+1}\le 2 q_k.
\eeq
Lemmas \ref{lemma_nonresonant_diophantine} and \ref{lemma_nonresonant_liouville} will correspond, respectively, to the cases  $s\le q_{k-k_0}^{C_d}$ (Diophantine to non-resonant transition) and $s>q_{k-k_0}^{C_d}$ (Liouville to non-resonant transition). In both cases, the arguments will be parallel to those in Lemmas \ref{lemma_diophantine_diophantine} and \ref{lemma_liouville_diophantine}, respectively. Let
$$
J_1:=[0;s q_{k-k_0}-1],
$$
$$
J:=-\lfloor sq_{k-k_0}/2\rfloor-s q_{k-k_0}+\l(J_1\cup(J_1+m)\r).
$$
By construction and the definition of $s$, $J$ contains $2s q_{k-k_0}$ distinct integer points. Let us consider separation properties for the points of $J\alpha$ modulo 1. Since $m$ is non-resonant and $|m|\le q_{k+1}/10$, without loss of generality $m=\ell q_k+m'$, where $\tau q_k\le m'\le (1-\tau) q_k$ and $0\le \ell q_k \le q_{k+1}/10$. As a consequence, 
\beq
\label{eq_lalpha}
\|\ell q_k\alpha\|=\ell \|q_k\alpha\|\le \frac{1}{10 q_k}.
\eeq
Note that the set $J_1\cup(J_1+m')$ still contains $2 s q_{k-k_0}$ distinct points, all of which are inside of an integer interval of length not exceeding $q_k$. As a consequence of Proposition \ref{prop_rotation_trajectory}, any two distinct points of $(J_0\cup(J_0+m'))\alpha$ are at least $\frac{1}{2 q_k}$-separated modulo $1$. Due to \eqref{eq_lalpha}, the change from $m$ to $m'$ can translate some of the points of $(J_1\cup(J_1+m'))\alpha$ at most by $\frac{1}{10 q_k}$. Thus, the points $J\alpha$ are (at least) $\frac{1}{3q_k}$-separated modulo 1.
\begin{remark}
\label{rem_sprime}
The same considerations would have worked with $s$ replaced by $s'=\lfloor s/10\rfloor$.
\end{remark}
\subsection{Diophantine to non-resonant transition: proof of Lemma \ref{lemma_nonresonant_diophantine}}
In the above notation, we have $s\le q_{k-k_0}^{C_d}$. For large $k$ depending on $\tau$, one can restate the conclusion of the previous section as that the points of $J\alpha$ are $q_{k-k_0}^{-C_d-2}$-separated modulo $1$. 

Then one can proceed similarly to Proposition \ref{prop_main_basecase} and Lemma \ref{lemma_diophantine_diophantine}: since $P_{2 s q_{k-k_0}-1}$ has at most $2 s q_{k-k_0}-1$ intersection points, there is a point $j\in J$ such that the $\frac13 q_{k-k_0}^{-C_d-2}$-neighborhood of $j\alpha$ does not contain any intersection points modulo 1. As a consequence, a similar application of Lemma \ref{lemma_sampling_points} with $l=t=0$ (no intersection points, one sampling point) will imply
$$
P_n^{<B_-(n)}(x+j\alpha)\ge 1,\quad n=2s q_{k-k_0}-1.
$$
After a similar application of Theorem \ref{th_main_cramer} and Proposition \ref{prop_origin}, will imply that every $m$ under consideration is $(L(E)-o(1),2 s q_{k-k_0}-1)$-regular for $\sigma<1/5$.

\subsection{Liouville to non-resonant transition: proof of Lemma \ref{lemma_nonresonant_liouville}}
With $s$ defined as above in Subsection \ref{subsection_nonresonant_preliminaries}, we now have $s>q_{k-k_0}^{C_d}$. Let $s':=\lfloor s/10\rfloor$. Proceed with the same steps as above, with $s$ replaced by $s'$:
$$
J_0:=[0;q_{k-k_0}-1],
$$
$$
J_1:=[0;s' q_{k-k_0}-1]=\cup_{r=0}^{s'-1}(J_0+r q_{k-k_0}),
$$
$$
J:=(J_1-s'q_{k-k_0}-\lfloor s' q_{k-k_0}/2\rfloor)\cup(J_1-s'q_{k-k_0}-\lfloor s' q_{k-k_0}/2\rfloor+m).
$$
In view of Remark \ref{rem_sprime}, $J$ consists of $2 s' q_{k-k_0}$ distinct points, and the corresponding points of $J\alpha$ are $\frac{1}{3 q_k}$-separated modulo $1$. The remaining part the proof of Lemma \ref{lemma_liouville_diophantine}, verbatim, with $\frac{1}{4 q_k}$ in \eqref{eq_verbatim} replaced by $\frac{1}{6 q_k}$.\,\,\qed

\subsection{Domination of resonant regions} Let $\psi$ be a generalized eigenfunction as in the proof of Theorem \ref{th_main} (Subsection \ref{subsection_proof_main}). In this short section, we will state an auxiliary lemma that would allow us to reduce estimates of the values $|\psi(m)|$ to those for resonant $m$. Recall the definition of resonant points, and denote
\beq
\label{eq_resonant_points_copy}
R_\ell:=[\ell q_k-b_k;\ell q_k+b_k], \quad r_{\ell}:=\max\{|\psi(m)|\colon m\in R_{\ell}\}=|\psi(m_{\ell})|,\quad m_{\ell}\in R_{\ell}.
\eeq
\begin{lemma}
\label{lemma_domination_resonant}
Let $m$ be non-resonant as follows:
$$
m\in [\ell q_k+b_k;(\ell+1)q_k-b_k],\quad |\ell|\le \frac{q_{k+1}}{10 q_k}.
$$
Then
\beq
\label{eq_decay_between_copy}
|\psi(m)|\le r_{\ell}e^{-(L(E)-o(1))\dist (m,R_{\ell})}+r_{\ell+1}e^{-(L(E)-o(1))\dist (m,R_{\ell+1})}.
\eeq
\end{lemma}
\begin{proof}
Follows from lemmas \ref{lemma_dominating} and \ref{lemma_terminating} in the setting of Remark \ref{rem_same_exponent}, with good intervals obtained by the means of Lemmas \ref{lemma_nonresonant_diophantine}, \ref{lemma_nonresonant_liouville} (two previous subsections).
\end{proof}
\subsection{Estimates in the resonant regions: proof of Lemma \ref{lemma_resonant}} 
Suppose that $m$ is resonant: that is,
$$
m\in [\ell q_k-b_k,\ell q_k+b_k],\quad 0<|\ell q_k|\le \frac{q_{k+1}}{10}.
$$
Let
$$
J_0:=[0;q_k-1];
$$
$$
J:=-\lfloor 3 q_k/2\rfloor+(J_0\cup(J_0+\ell q_k)).
$$
Under the above assumptions on $\ell$, the set $J$ has $2 q_k$ distinct integer points. By Proposition \ref{prop_rotation_trajectory}, the points of $J_0\alpha$ are $\frac{1}{2 q_k}$-separated modulo $1$. The set $(J_0+\ell q_k)\alpha$ is a translation of $J_0\alpha$ by $\ell q_k\alpha$. Note that 
$$
\frac{1}{q_{k+1}}<\|\ell q_k\alpha\|=\ell \|q_k\alpha\|\le \frac{1}{10 q_k}.
$$
As a consequence, the points of $J\alpha$ can be split into $q_k$ pairs. The distance between two points in the same pair (modulo 1) is equal to $|\ell| \|q_k\alpha\|$. The distance between two points in different pairs is at least $\frac{1}{2 q_k}-\frac{1}{10 q_k}\ge \frac{1}{3 q_k}$.

In view of Proposition \ref{prop_rotation_trajectory}, we have
$$
J_0\alpha\subset\bigcup_{j=0}^{q_k-1}\l[\frac{j}{q_{k}}-\frac{1}{100 q_{k}},\frac{j}{q_{k}}+\frac{1}{100 q_{k}}\r]\,\, \mathrm{mod}\,\,1,
$$
and therefore
$$
(J_0\cup(J_0+m))\alpha\subset\bigcup_{j=0}^{q_k-1}\l[\frac{j}{q_{k}}-\frac{1}{9 q_{k}},\frac{j}{q_{k}}+\frac{1}{9 q_{k}}\r]\,\, \mathrm{mod}\,\,1.
$$
Being a translation of the above set, the set $J\alpha$ is contained in a union of $q_k$ equally spaced modulo $1$ intervals of length $\frac{2}{9 q_k}$. Each interval contains exactly two points of $J\alpha$ (with distance between these points being $|\ell|\|q_k\alpha\|$). As before, we will call these intervals {\it sampling intervals}. Note that the $\frac{1}{9 q_k}$-neighborhoods of the sampling intervals are still disjoint modulo $1$.

Consider the determinant $P_{2q_k-1}(x)$. Since it has at most $2 q_k-1$ intersection points, there exists at least one sampling interval $I$ whose $\frac{1}{9 q_k}$-neighborhood $I'$ contains at most one intersection point. Apply Lemma \ref{lemma_sampling_points} with $d=|\ell| \|q_k\alpha\|$, $k=\ell=1$ (two sampling points, one intersection point). We thus obtain
$$
P_{n}(x+j\alpha)^{<B_-(n)}\ge |\ell|\|q_k\alpha\|,\quad n=2 q_k-1.
$$
We will also estimate (using Proposition \ref{prop_rotation_trajectory} and the definition of $\beta$:
$$
|\ell|\|q_k\alpha\|\ge \frac{|\ell|}{2 q_{k+1}}\ge e^{-(\beta+o(1)) q_k+\log |\ell|}.
$$
Overall, this implies after applying Theorem \ref{th_ldt} and absorbing several non-essential factors into $o(1)$:
$$
|P_{2 q_k-1}(x+j\alpha)|\ge e^{\log |\ell|+2 q_k(L(E)-\beta/2-o(1))}.
$$
Apply Theorem \ref{th_main_cramer}. Note that this is a somewhat unusual application since, in some cases, the matrix elements of the Green's function will be exponentially large and not exponentially small. However, they will be compensated later. For $m\in J':=[j;j+2 q_k-2]$ and every solution of the eigenvalue equation $\psi$ we have the following bound:
\begin{multline}
\label{eq_poisson_resonant1}
|\psi(m)|\le e^{-\log |\ell|-2(L(E)-\beta/2-o(1))q_k+L(E)(2 q_k+j-m)}|\psi(j-1)|\\
+e^{-\log |\ell|-2(L(E)-\beta/2-o(1))q_k+L(E)(m-j)|}|\psi(j+2 q_k-1)|\\
\le e^{-\log |\ell|+(\beta+o(1))q_k}\left(e^{-L(E)(m-j)}|\psi(j-1)|+e^{-L(E)(2 q_k+j-m)}|\psi(j+2 q_k-1)|\right).
\end{multline}
Since the length of $J'$ is $2 q_k-1$ and due to the construction of $J$, at least one resonant region $R_s$ (see \eqref{eq_resonant_points_copy}) satisfies the property 
$$
R_s\subset J',\quad \dist(R_s,\partial J')>\frac{q_k}{2}(1-4\tau).
$$
In fact, this will always be true with $s=0$ or $s=\ell$. Apply \eqref{eq_poisson_resonant1} to the point $m_s\in R_s$ where $|\psi|$ attains its maximal value (see \eqref{eq_resonant_points_copy}). It is easy to see that the corresponding points $j-1$ and $j+2 q_k-1$ are located strictly between $R_{s-2}$ and $R_{s+2}$. For either of these points that are non-resonant, apply \eqref{eq_decay_between_copy}. After absorbing several factors into $o(1)$, we finally arrive to
\beq
\label{eq_decay_resonances}
r_s\le e^{-\log |\ell|-(L(E)(1-2\tau)-\beta-o(1))q_k}(r_{s+1}+r_s+r_{s-1})+e^{-\log |\ell|-(2 L(E)(1-2\tau)-\beta-o(1))q_k}(r_{s+2}+r_{s-2}).
\eeq
Choose $\tau>0$ small enough so that 
\beq
\label{eq_tau_small}
L(E)(1-2\tau)-\beta>0
\eeq
With this choice, one can move $r_s$ from the right hand side to the left hand side. For every $m\in R_s$ this implies either
$$
|\psi(m)|\le e^{-\log |\ell|-(L(E)(1-2\tau)-\beta-o(1))q_k}|\psi(m_{s+1})|+e^{-\log |\ell|-(L(E)(1-2\tau)-\beta-o(1))q_k}|\psi(m_{s-1})|
$$
or
$$
|\psi(m)|\le e^{-\log |\ell|-(2 L(E)(1-2\tau)-\beta-o(1))q_k}\psi(m_{s-2})|+e^{-\log |\ell|-(2 L(E)(1-2\tau)-\beta-o(1))q_k}|\psi(m_{s+2})|.
$$
Here we used the inequality $|a+b|\le 2\max\{|a|,|b|\}$ and absorbed the factor $2$ into $o(1)$.

In view of \eqref{eq_tau_small}, the first inequality would imply that the interval $[m_{s-1},m_{s+1}]$ is $(L(E)(1-2\tau)-\beta,D)$-good for $\psi$, where $|D-2 q_k|\le 2\tau q_k$. In that case, $\psi$ is $(L(E)(1-2\tau)-\beta,D)$-regular at $m$ with $|D-2 q_k|\le 2\tau q_k$. In case of second inequality, $\psi$ is $(L(E)(1-2\tau)-\beta/2,D)$-regular at $m$ with $|D-4 q_k|\le 8\tau q_k$.

From Proposition \ref{prop_origin}, we cannot have $s=0$, and
therefore $s=\ell$, and we have established regularity of every $m\in
R_{\ell}$.\,\,\qed

\section*{Acknowledgements} The work of S. J. was partially supported by NSF DMS--2052899, DMS--2155211, and Simons 896624. The work of I. K. was supported by  NSF DMS--1846114, DMS--2052519, and the 2022 Sloan Research Fellowship.

\end{document}